\documentclass[11pt]{amsart}
\usepackage{amsmath,amssymb}
\usepackage{framed}

\usepackage{fullpage}
\usepackage[dvips]{graphicx}
\usepackage{amsfonts,amsthm,url,tabularx, array}
\usepackage{enumerate,verbatim,float} 
\usepackage{tikz} 
\usepackage{complexity}

\usepackage{thmtools}
\usetikzlibrary{arrows} 
\def\COMMENT#1{}
\def\TASK#1{}
\def\noproof{{\unskip\nobreak\hfill\penalty50\hskip2em\hbox{}\nobreak\hfill%
        $\square$\parfillskip=0pt\finalhyphendemerits=0\par}\goodbreak}
\def\endproof{\noproof\bigskip}
\newdimen\margin   
\def\textno#1&#2\par{%
    \margin=\hsize
    \advance\margin by -4\parindent
           \setbox1=\hbox{\sl#1}%
    \ifdim\wd1 < \margin
       $$\box1\eqno#2$$%
    \else
       \bigbreak
       \hbox to \hsize{\indent$\vcenter{\advance\hsize by -3\parindent
       \sl\noindent#1}\hfil#2$}%
       \bigbreak
    \fi}

\def\eps{\varepsilon}

\def\LL{\mathcal{L}}

\newtheorem{thm}{Theorem}
\newtheorem{fact}[thm]{Fact}
\newtheorem{lemma}[thm]{Lemma}
\newtheorem{prop}[thm]{Proposition}
\newtheorem{cor}[thm]{Corollary}

\newtheorem*{claim}{Claim}

\newtheorem{obs}[thm]{Observation}
\newtheorem*{problem}{Problem}
\newtheorem*{adef}{Definition}
\def\COMMENT#1{}
\let\COMMENT=\footnote

\declaretheoremstyle[notefont=\bfseries,notebraces={}{},%
    headpunct={},postheadspace=1em]{mystyle}
\declaretheorem[style=mystyle,numbered=no,name=Theorem]{thm-hand}

\newcommand{\leqfptT}{ $\leq^{\textup{fpt}}_{\textup{T}}$ }
\newcommand{\paramcount}[1]{\textup{\textbf{p-\#}}\textsc{#1}}
\newcommand{\paramdec}[1]{\textup{\textbf{p-}}\textsc{#1}}
\newcommand{\size}[1]{\textrm{size}(#1)}
\newcommand{\countprob}{\textsc{\#Given-Size $\LL$-Free Subset}}

\DeclareMathOperator{\opt}{opt}

\allowdisplaybreaks

\title{On the complexity of finding and counting solution-free sets of integers}
\author{Kitty Meeks and  Andrew Treglown}
\thanks{The first author is supported by a Personal Research Fellowship from the Royal Society of Edinburgh, funded by the Scottish Government, and the second author is supported by EPSRC grant EP/M016641/1.}
\begin{document}
\label{firstpage}

\begin{abstract} 
Given a linear equation $\LL$, a set $A $ of integers
is $\LL$-free if $A$ does not contain any `non-trivial' solutions to $\LL$. This notion incorporates many central topics in combinatorial number theory such as sum-free and progression-free sets.
In this paper we initiate the study of (parameterised) complexity questions involving $\LL$-free sets of integers. The main questions we consider involve deciding whether  a finite set of integers $A$ has an $\LL$-free subset of a given size, and counting all such $\LL$-free subsets. We also raise a number of open problems.
\end{abstract}
\date{\today}
\maketitle

\section{Introduction}

Sets of integers which do not contain any solutions to some linear equation have received a lot of attention in the field of combinatorial number theory.  Two particularly well-studied examples are \emph{sum-free sets} (sets avoiding solutions to the equation $x + y = z$) and \emph{progression-free sets} (sets that do not contain any 3-term arithmetic progression $x,y,z$ or equivalently avoid solutions to the equation $x + z = 2y$).  A lot of effort has gone into determining the size of the largest solution-free subset of $\{1,\ldots,n\}$ and other sets of integers, and into computing (asymptotically) the number of (maximal) solution-free subsets of $\{1,\ldots,n\}$.  

In this paper we initiate the study of the computational complexity of problems involving solution-free subsets.  We are primarily concerned with determining the size of the largest subset of an arbitrary set of integers $A$ which avoids solutions to a specified linear equation $\LL$; in particular, we focus on  sum-free and progression-free sets, but many of our results also generalise to larger families of linear equations.  For suitable equations $\LL$, we demonstrate that the problem of deciding whether $A$ contains a solution-free subset of size at least $k$ is $\NP$-complete (see Section~\ref{sec:decision});  we further show that it is hard to approximate the size of the largest solution-free subset within a factor $(1+\epsilon)$ (see Section~\ref{sec:ap}), or to determine for a constant $c < 1$ whether $A$ contains a solution-free subset of size at least $c|A|$ (see Section~\ref{sec:6}).  On the other hand, in Section~\ref{sec:param-dec} we see that the decision problem is fixed-parameter tractable when parameterised by either the cardinality of the desired solution-free set, or by the number of elements of $A$ we can exclude from such a set.  We also consider the complexity, with respect to various parameterisations, of counting the number of solution-free sets of a specified size (see Section~\ref{sec:counting}): while there is clearly no polynomial-time algorithm in general, the problem is fixed-parameter tractable when parameterised by the number of elements we can exclude from $A$; we show that there is unlikely to be a fixed-parameter algorithm to solve the counting problem exactly when the size of the solution-free sets is taken as the parameter, but we give an efficient approximation algorithm for this setting.  Finally, in Section~\ref{sec:ex} we consider all of these questions in a variant of the problem, where we specify that a given solution-free subset $B \subset A$ must be included in any solution.

Many of our results are based on the fact that we can set up polynomial-time reductions in both directions between our problem and (different versions of) the well-known hitting set problem for hypergraphs.  We also derive some new lower-bounds on the size of the largest solution-free subset of an arbitrary set of integers for certain equations $\LL$, which may be of independent interest.

In the remainder of this section, we give some background on solution-free sets in Section~\ref{sec:background} and review the relevant notions from the study of computational complexity in Section~\ref{sec:complexity}.

\subsection{Background on solution-free sets}
\label{sec:background}

Consider a fixed linear equation
$\LL$ of the form 
\begin{align}\label{Leq}
a_1x_1+\dots +a_\ell x_\ell =b
\end{align}
where $a_1,\dots, a_\ell,b \in \mathbb Z$.
We say that $\LL$ is \emph{homogeneous} if $b=0$.
If 
$$\sum _{i \in [k]} a_i=b=0$$ then we say that $\LL$ is \emph{translation-invariant}. (Here $[k]$ denotes the set $\{1,\ldots,k\}$.)
Let $\LL$ be translation-invariant. Then notice that $(x,\dots, x)$ is a `trivial' solution of (\ref{Leq}) for any $x$.
More generally, a solution $(x_1, \dots , x_k)$ to $\LL$ is said to be \emph{trivial} if there exists a partition $P_1, \dots ,P_\ell$ of $[k]$ so that: 
\begin{itemize}
\item[(i)] $x_i=x_j$ for every $i,j$ in the same partition class $P_r$; 
\item[(ii)] For each $r \in [\ell]$, $\sum _{i \in P_r} a_i=0$. 
\end{itemize}
A set $A $ of integers is \emph{$\LL$-free} if $A$ does not contain any non-trivial solutions to $\LL$. If the equation $\LL$ is clear from the context, then we simply say $A$ is \emph{solution-free}.

\subsubsection{Sum-free sets}
A set $S$ (of integers or elements of a group) is \emph{sum-free} if there does not exist $x,y,z$ in $S$ such that $x+y=z$.
The topic of sum-free sets has a rich history spanning a number of branches of mathematics. 
In 1916 Schur~\cite{schur} proved that, given $r \in \mathbb N$, if $n$ is sufficiently large, then any $r$-colouring of  $[n]:=\{1, \dots , n\}$ yields a monochromatic triple $x,y,z$ such that $x+y=z$. (Equivalently, 
$[n]$ cannot be partitioned into $r$ sum-free sets.)
 This theorem was followed by other seminal related results such as van der Waerden's theorem~\cite{vdw}, and ultimately led to the birth of arithmetic Ramsey theory.

 Paul Erd\H{o}s had a particular affinity towards sum-free sets. In 1965 he~\cite{erdos} proved one of the cornerstone results in the subject: every set of $n$ non-zero integers $A$ contains a sum-free subset of
size at least $n/3$. Employing the probabilistic method, Alon and Kleitman~\cite{alonkleitman} improved this bound to $(n+1)/3$ and further, Kolountzakis~\cite{kolo} gave a polynomial time algorithm for constructing such a sum-free subset. 
Then, using a Fourier-analytical approach, Bourgain~\cite{bourg} further improved the bound to $(n+2)/3$ in the case when $A$ consists of positive integers.
Erd\H{o}s~\cite{erdos} also raised the question of determining upper bounds for this problem: recently Eberhard, Green and Manners~\cite{egm} asymptotically resolved this important classical problem by proving that there is a set of positive integers $A$ of size $n$ such that $A$ does not contain any sum-free subset of size greater than $n/3+o(n)$.

A number of important questions concerning sum-free sets were raised in two papers of Cameron and Erd\H{o}s~\cite{cam1, CE}. In~\cite{cam1}, Cameron and Erd\H{o}s conjectured that there are $\Theta (2^{n/2})$
sum-free subsets of $[n]$. Here, the lower bound follows by observing that the largest sum-free subset of $[n]$ has size $\lceil n/2 \rceil $; this is attained by the set of odds in $[n]$ and by $\{\lfloor n/2 \rfloor+1, \dots , n\}$. Then, for example, by taking all subsets of $[n]$ containing only odd numbers one obtains at least $2^{n/2}$ sum-free subsets of $[n]$.
After receiving much attention, the Cameron--Erd\H{o}s conjecture was proven independently by Green~\cite{G-CE} and Sapozhenko~\cite{sap}. 
Given a set $A$ of integers we say $S\subseteq A$ is a \emph{maximal sum-free subset of $A$} if $S$ is sum-free and it is not properly contained in another sum-free subset of $A$.
Cameron and Erd\H{o}s~\cite{CE} raised the question of how many maximal sum-free subsets there are in $[n]$. Very recently, this question has been resolved via a combinatorial approach by Balogh, Liu, Sharifzadeh and Treglown~\cite{BLST,BLST2}.

Sum-free sets have also received significant attention with respect to groups. One highlight in this direction is work of Diananda and Yap~\cite{yap} and Green and Ruzsa~\cite{GR-g} that 
determines the size of the largest sum-free subset for every finite abelian group. In each case the largest sum-free set has size linear in the size of the abelian group. 
Another striking result in the area follows from Gowers' work on quasirandom groups. 
Indeed, Gowers~\cite{gow} proved that there are non-abelian groups for which the largest sum-free subset has sublinear size, thereby answering a question of Babai and S\'os~\cite{sos}. See the survey of Tao and Vu~\cite{taovu} for a discussion on further problems concerning sum-free sets in groups.

\subsubsection{Progression-free sets}
A set $S$ (of integers or elements of a group) is \emph{progression-free} if there does not exist distinct $x,y,z$ in $S$ such that $x+y=2z$. 
The study of progression-free sets has focused on similar questions to those relating to sum-free sets. 

Unlike in the case of sum-free sets, one cannot ensure that every  finite set of non-zero integers contains a progression-free subset of linear size. Indeed, a classical result of Roth~\cite{roth} implies that the largest progression-free subset of $[n]$ has
size $o(n)$. This has led to much interest in determining good bounds on the size of such a subset of $[n]$. See~\cite{bloom,elk,greenwolf} for the state-of-the-art lower and upper bounds for this problem. 

Roth's theorem has been generalised in various directions; most famously via Szemer\'edi's theorem~\cite{sz} which ensures that, if $n$ is sufficiently large,  every subset of $[n]$ of linear size contains arithmetic progressions of arbitrary length. 
Analogues of Roth's theorem have also been considered for finite abelian groups; see, for example, \cite{brown, fgr2, lev}.

As in the case of sum-free sets, it is also natural to ask for the number of progression-free subsets of $[n]$. More generally, Cameron and Erd\H{o}s~\cite{cam1} raised the question of how many subsets of $[n]$ do not contain an arithmetic progression  of length $k$. Significant progress on the problem has recently been made in~\cite{bls, container1, container2}.

\smallskip

We remark that there has also been much work on $\LL$-free sets other than the cases of sum-free and progression-free sets; see for example~\cite{ruzsa, ruzsa2}.

\subsection{Computational complexity}
\label{sec:complexity}

In this paper we are concerned with determining which problems related to solution-free sets of integers are computationally tractable.  In the first instance, we seek to classify decision problems as either belonging to the class $\P$ (i.e. being solvable in polynomial time) or being $\NP$-hard and so unlikely to admit polynomial-time algorithms.  For further background on computational complexity, the classes $\P$ and $\NP$, and polynomial-time reductions we refer the reader to \cite{garey}.

When dealing with sets of positive integers as input, it should be noted that the amount of space required to represent the input depends both on the cardinality of the set and on the magnitude of the numbers in the set.  Given any finite set $A\subseteq \mathbb Z$, we write 
$\max (A)$ and $\min (A)$ for the elements of $A$ whose  values are maximum and minimum respectively; we further define $\max ^*(A):=\max\{|a| :a \in A\}$ and $\min ^*(A):=\min\{|a| :a \in A\}$.  We write
$\size{A}$ for the number of bits required to represent $A$, and note that there exist positive constants $c_1$ and $c_2$ such that
$$c_1 \max\{|A| \log({\min} ^* (A)), \log({\max} ^* (A))\} \leq \size{A} \leq c_2 |A| \log ({\max} ^* (A)).$$
We therefore consider a problem involving the set $A$ to belong to $\P$ if it can be solved by an algorithm whose running time is bounded by a polynomial function of $\size{A}$; note that this is true if and only if the running time is bounded by a polynomial function of $|A| \log ( \max^*(A))$.  If $A = \{a_1,\ldots,a_n\}$, we will assume that $A$ is stored in such a way that, given $i$, we can read the element $a_i$ in time $\mathcal{O}(\log (|a_i|))$.

There are two basic operations which we will need to consider in almost all of the algorithms and reductions discussed in this paper.  First of all, we often need to determine whether a given $\ell$-tuple $(x_1,\ldots,x_{\ell})$ is a solution to the $\ell$-variable linear equation $\LL$.  Note that we assume throughout that $\ell$ and all coefficients in $\LL$ are constants, but $x_1,\ldots,x_{\ell}$ are taken to be part of the input; thus, allowing for the time required to carry out the necessary arithmetic operations, we can certainly determine whether $(x_1,\ldots,x_{\ell})$ is a solution in time $\mathcal{O}(\log (\max_{1 \leq i \leq \ell} |x_i|))$.

Secondly, we will in many cases need to determine whether a given set $A \subseteq \mathbb{Z}$ is $\LL$-free.  We can do this in a naive way by considering all possible $\ell$-tuples, and checking for each one whether it is a solution.  Since there are $|A|^{\ell}$ possible $\ell$-tuples, we see by the reasoning above that we can complete this procedure in time $\mathcal{O}\left(|A|^{\ell} \cdot \log^2 \left(\max ^*(A)\right)\right)$.  Note that, as $\ell$ is a constant, this is a polynomial-time algorithm in terms of $\size{A}$.

\smallskip

In this paper, we will also discuss the parameterised complexity of various decision problems involving solution-free sets.  Parameterised complexity provides a multivariate framework for the analysis of hard problems: if a problem is known to be $\NP$-hard, so that we expect the running-time of any algorithm to depend exponentially on some aspect of the input, we can seek to restrict this exponential blow-up to one or more \emph{parameters} of the problem rather than the total input size.  This has the potential to provide an efficient solution to the problem if the parameter(s) in question are much smaller than the total input size.  A parameterised problem with total input size $n$ and parameter $k$ is considered to be tractable if it can be solved by a so-called \emph{fpt-algorithm}, an algorithm whose running time is bounded by $f(k)\cdot n^{\mathcal{O}(1)}$,  where $f$ can be any computable function.  Such problems are said to belong to the complexity class $\FPT$.  The primary method for showing that a problem is unlikely to belong to $\FPT$ is to show that it is hard for some class $W[t]$ (where $t \geq 1$) in the W-heirarchy (see \cite{flumgrohe} for a formal definition of these classes).  When reducing one parameterised problem to another in order to demonstrate the hardness of a parameterised problem, we have to be a little more careful than with standard $\NP$-hardness reductions: as well as making sure that we can construct the new problem instance efficiently, we also need to ensure that the parameter value in the new problem depends only on the parameter value in the original problem.

Let $\Sigma$ be a finite alphabet.  We define a \emph{parameterised decision problem} to be a pair $(\Pi,\kappa)$ where $\Pi: \Sigma^* \rightarrow \{\text{YES, NO}\}$ is a function  and $\kappa: \Sigma^* \rightarrow \mathbb{N}$ is a parameterisation (a polynomial-time computable mapping).  Then, an \emph{fpt-reduction from $(\Pi,\kappa)$ to $(\Pi',\kappa')$} is an algorithm $A$ such that
\begin{enumerate}
\item $A$ is an fpt-algorithm;
\item given a yes-instance of $\Pi$ as input, $A$ outputs a yes-instance of $\Pi'$, and given a no-instance of $\Pi$ as input, $A$ outputs a no-instance of $\Pi'$;
\item there is a computable function $g$ such that, if $I$ is the input to $A$, then $\kappa'\left(A(I)\right) \leq g\left(\kappa(I)\right)$.
\end{enumerate}
For further background on the theory of parameterised complexity we refer the reader to \cite{downeyfellows13,flumgrohe}.  We will also consider parameterised counting problems in Section \ref{sec:counting}, and relevant notions will be introduced at the start of this section.


\section{The decision problem}
\label{sec:decision}

In this section, we consider the following problem, where $\LL$ is any fixed linear equation.
\begin{framed}
\noindent $\LL$-\textsc{Free Subset}\newline
\textit{Input:} A finite set $A \subseteq \mathbb Z$ and $k \in \mathbb N$.\newline
\textit{Question:} Does there exist an $\LL$-free subset $A' \subseteq A$ such that $|A'|=k$?
\end{framed}

We show that this problem is closely related to the well-known \textsc{Hitting Set} problem, and exploit this relationship to show that the problem is $\NP$-complete for a large family of equations $\LL$, including those defining both sum-free and progression-free sets.  On the other hand, we see that the problem is polynomially solvable whenever $\LL$ is a linear equation with only two variables.

We begin in Section \ref{sec:reduction} by showing how to construct an instance of $\LL$-\textsc{Free Subset} that corresponds in a specific way to a given hypergraph; we will make use of this same construction to prove many results in later sections of the paper.  We exploit this construction to give an $\NP$-completeness proof in Section~\ref{sec:dec-NP}, before considering the two-variable case in \ref{sec:dec-2var}.

\subsection{A useful construction}
\label{sec:reduction}

In this section we describe the main construction we will exploit throughout the paper, and prove its key properties.  Recall that an $\ell$-uniform hypergraph is a hypergraph in which every edge has size exactly $\ell$.  Given any set $X$ and $\ell \in \mathbb{N}$, we write $X^{\ell}$ for the set of ordered $\ell$-tuples whose elements belong to $X$.

\begin{lemma}\label{useful1} Let $\LL$ be a linear equation $a_1x_1 + \cdots + a_{\ell}x_{\ell} = by$ where each $a_i \in \mathbb{N}$ and $b \in \mathbb{N}$ are fixed, and let $H=(V,E)$ be an $\ell$-uniform hypergraph.  Then we can construct in polynomial time a set $A \subseteq \mathbb{N}$ with the following properties:\begin{enumerate}
\item $A$ is the disjoint union of two sets $A'$ and $A''$, where $|A'| = |V|$ and $|A''|=|E|$;
\item there exist bijections $\phi_V: A' \rightarrow V$ and $\phi_E: A'' \rightarrow E$;
\item for every $(x_1,\ldots,x_{\ell},y) \in A^{\ell + 1}$, we have that $(x_1,\ldots,x_{\ell},y)$ is a non-trivial solution to $\LL$ if and only if $x_1,\ldots,x_{\ell} \in A'$, $y \in A''$ and $\{\phi_V(x_1),\ldots,\phi_V(x_{\ell})\} = \phi_E(y)$;
\item $\log (\max(A))=\mathcal{O}(|V|)$.
\end{enumerate}
\end{lemma}
\proof
Write $a:=\max _{1\leq j\leq \ell} a_j$ and 
set $d:=2\ell  a^2b^2$.
Let $V=:\{v_1, \dots, v_n \}$ denote the vertex set of $H$. Define 
$$A':=\{bd^i : \ i \in [n]\} $$
 and 
$$A'':=\left \{ (a_1 d^{i_1}+a_2d^{i_2}+\dots +a_\ell d^{i_{\ell}}) : i_1<i_2<\dots<i_\ell \text{ and } v_{i_1}v_{i_2}\dots v_{i_{\ell}} \in E\right \}.$$
Further, define $\phi_V: A' \rightarrow V$ by setting $\phi_V(bd^i)=v_i$ for all $i \in [n]$, and note that $\phi_V$ is a well-defined bijection.  We define $\phi_E: A'' \rightarrow E$ by setting
$\phi _E(a_1 d^{i_1}+a_2d^{i_2}+\dots +a_\ell d^{i_{\ell}})= v_{i_1}v_{i_2}\dots v_{i_{\ell}}$ where $i_1<i_2<\dots<i_\ell$ and  $v_{i_1}v_{i_2}\dots v_{i_{\ell}} \in E$; to see that $\phi_E$ is also a well-defined bijection it suffices to observe that, by the uniqueness of base-$d$ representation of natural numbers, there is a unique way to write any $y \in A''$ in the form $a_1 d^{i_1}+a_2d^{i_2}+\dots +a_\ell d^{i_{\ell}}$.  It follows from the bijectivity of $\phi_V$ and $\phi_E$ that we have defined $A'$ so that for each vertex $v_i\in V$ there is a unique number $bd^i\in A'$ associated with it, and defined $A''$ so that for each edge $v_{i_1}v_{i_2}\dots v_{i_{\ell}}\in E$ there is a unique number 
$(a_1 d^{i_1}+a_2d^{i_2}+\dots +a_\ell d^{i_{\ell}})\in A''$ associated with it. Define $A:=A'\cup A''$.
Given $H$ we produce $A$ in time $\mathcal{O}((|V|+|E|)\log bd^{|V|})=\mathcal{O}(n^{\ell+1})$.

Notice that conditions (1), (2) and (4) of the lemma immediately hold. To prove (3) note that it suffices to prove the following claim.

\begin{claim}\label{c1}
The only non-trivial solutions $(x_1,x_2,\dots , x_{\ell}, y)$ to $\LL$ in $A$ are such that each $x_j=bd^{i_j}\in A'$ for some $i_1<i_2<\dots<i_{\ell}$ and  $y=(a_1 d^{i_1}+a_2d^{i_2}+\dots +a_\ell d^{i_{\ell}})\in A''$.
\end{claim}
\noindent To prove the claim it is helpful to consider the natural numbers working in base $d$. We will use the \emph{coordinate notation} $[c_0,c_1,c_2,\dots]$ to denote the natural number $c_0d^0+c_1d^1+c_2d^2+\dots$.
So with respect to this notation, each $bd^i\in A'$ has a zero in each coordinate except the $i$th coordinate, which takes value $b$. Each 
$(a_1 d^{i_1}+a_2d^{i_2}+\dots +a_\ell d^{i_{\ell}})\in A''$ takes value $a_j$ in its $i_j$th coordinate, and zero otherwise.

Suppose we have $x_1, \dots , x_t \in A$ for some $t \in \mathbb N$. Define $\text{coord}(x_1,\dots, x_t)$ to be the set of all integers $i\geq 0$ such that for at least one of the elements $x_j$ in $\{x_1,\dots,x_t\}$, the $i$th coordinate of $x_j$ is non-zero.  Note that (since $d$ is sufficiently large compared with the $a_i$ and $b$) we have $|\text{coord}(x)|=1$ for all $x \in A'$ and $|\text{coord}(y)| = \ell$ for all $y \in A''$.  Moreover, for any $x,x' \in A$, $\text{coord}(x) = \text{coord}(x')$ if and only if $x = x'$.

Suppose $(x_1,\dots,x_{\ell},y)$ is a non-trivial solution to $\LL$ in $A$. (Note the choice of $\LL$ ensures the only trivial solutions to $\LL$ are such that $x_1=\dots=x_\ell=y$.)
Crucially we defined $d$ to be large with respect to the $a_i$ and $b$. Thus, $\text{coord}(x_1,\dots, x_{\ell})=\text{coord}(a_1 x_1+a_2x_2+\dots +a_\ell x_{\ell})$.
That is, in coordinate notation, the coordinates of $(a_1 x_1+a_2x_2+\dots +a_\ell x_{\ell})$ that are non-zero are precisely those coordinates that are non-zero in at least one of the $x_j$.
So this gives us that $\text{coord}(x_1,\dots, x_{\ell})=\text{coord}(by)$.

If $y \in A'$ then $|\text{coord}(by)|=1$. Note in this case we obtain a contradiction if  $|\text{coord}(x_1,\dots, x_{\ell})|\geq 2$. So it must be the case that $x_1=\dots=x_\ell$ and $x_1 \in A'$. This means that there is some $i \in [n]$ such that
$x_1=\dots=x_\ell=bd^i$;  $y=bd^i$; and further $a_1+\dots +a_\ell=b$. Thus $(x_1,\dots,x_\ell,y)$ is a trivial solution to $\LL$, a contradiction to our assumption.

Therefore $y \in A''$. Write $y=a_1 d^{i_1}+a_2d^{i_2}+\dots +a_\ell d^{i_{\ell}}$ where  $i_1<i_2<\dots<i_\ell$.
Suppose that  all of the $x_j$ lie in $A''$. Then since  $\text{coord}(x_1,\dots, x_{\ell})=\text{coord}(by)$, we must have that $x_1=\dots=x_\ell=y$. This implies $a_1+\dots +a_\ell=b$ and hence $(x_1,\dots,x_\ell,y)$ is a trivial solution to $\LL$, a contradiction.

Next suppose there is at least one $x_j \in A''$ and at least one $x_{j'} \in A'$. 
Without loss of generality, we may assume that there is some $1\leq \ell '\leq \ell -1$ so that $x_1,\dots, x_{\ell '}\in A''$ and $x_{\ell'+1},\dots, x_{\ell}\in A'$.
Since $\text{coord}(x_1,\dots, x_{\ell})=\text{coord}(by)$, we must have that $x_1=\dots=x_{\ell '}=y$. Furthermore, as $a_1x_1+\dots+a_\ell x_\ell = b y$ we have that $a_1x_1+\dots+a_{\ell'} x_{\ell'} < b y$ and thus
$a_1+\dots+a_{\ell'}<b$. In particular, $\text{coord}((b-a_1-\dots-a_{\ell'})y)=\{i_1,\dots, i_{\ell}\}$.
Therefore, as $a_{\ell'+1}x_{\ell'+1}+\dots+a_{\ell}x_{\ell}=(b-a_1-\dots-a_{\ell'})y$, this implies
$\text{coord}(a_{\ell'+1}x_{\ell'+1}+\dots+a_{\ell}x_{\ell})=\{i_1,\dots, i_{\ell}\}$.
However, since $\ell'\geq 1$ and $x_{\ell'+1},\dots, x_{\ell} \in A'$, we have that $|\text{coord}(a_{\ell'+1}x_{\ell'+1}+\dots+a_{\ell}x_{\ell})|<\ell$, a contradiction.

Altogether this implies that each $x_j \in A'$. The claim, and therefore lemma, now immediately follows. 
\endproof

A relationship between independent sets in  $H$ and $\LL$-free subsets of $A$ now follows easily.

\begin{cor}
\label{cor:is-reduction}
Let $\LL$ be a linear equation $a_1x_1 + \cdots + a_{\ell}x_{\ell} = by$ where each $a_i \in \mathbb{N}$ and $b \in \mathbb{N}$ are fixed, and let $H=(V,E)$ be an $\ell$-uniform hypergraph. 
Let $A$ and $A''$ be as in Lemma~\ref{useful1} on input $H$ and $\LL$.
Then, for any $k \in \mathbb N$,  there is a one-to-one correspondence between independent sets of $H$ of cardinality $k$ and the $\LL$-free subsets of $A$ of cardinality $|A''|+k$ which contain all the elements of $A''$. 
\end{cor}
\proof
The corollary follows immediately from Lemma~\ref{useful1}. Indeed, let $\phi_V$ and $\phi _E$ be as in Lemma~\ref{useful1}. Given an independent set $I$ of $H$, note that $\phi ^{-1}_V(I) \cup A''$ is an $\LL$-free subset of $A$ of size
$|I|+|A''|$; by bijectivity of $\phi_V$, the $\LL$-free subsets corresponding to independent sets $I_1 \neq I_2$ are distinct. Further, given any $\LL$-free subset $S\cup A''$ of $A$ of size $|S|+|A''|$, we have that $\phi _V (S)$ is an independent set in $H$ of size $|S|$; again, by bijectivity of $\phi_V$, we obtain a unique independent set for each such $\LL$-free subset.
\endproof

Finally, if we are only interested in the existence of independent sets, we can drop one of the conditions on the $\LL$-free subsets.

\begin{cor}
\label{cor:includes-A''}
Let $\LL$ be a linear equation $a_1x_1 + \cdots + a_{\ell}x_{\ell} = by$ where each $a_i \in \mathbb{N}$ and $b \in \mathbb{N}$ are fixed, and let $H=(V,E)$ be an $\ell$-uniform hypergraph. 
Let $A$ and $A''$ be as in Lemma~\ref{useful1} on input $H$ and $\LL$.
Then, for any $k \in \mathbb N$,  $H$ contains an independent set of cardinality $k$ if and only if $A$ contains an $\LL$-free subset of cardinality $|A''| + k$.
\end{cor}
\proof
By Corollary \ref{cor:is-reduction}, it suffices to show that, if $A$ contains an $\LL$-free subset of cardinality $|A''| + k$, then in fact $A$ contains such a subset which includes all elements of $A''$.  To see this, let $A_1$ be an $\LL$-free subset of $A$ of size $|A''| + k$ which does not contain all elements of $A''$; we will show how to construct an $\LL$-free subset of equal or greater size which does have this additional property.

Suppose $y\in A''$ such that $y \not \in A_1$. By Lemma~\ref{useful1}(3)  there is a unique choice of $x_1,\dots,x_\ell \in A$ such that $(x_1,\dots,x_\ell,y)$ is a non-trivial solution to $\LL$ in $A$.
If one of these $x_j$ does not lie in $A_1$ we add $y$ to $A_1$ without creating a solution to $\LL$. Otherwise, arbitrarily remove one of the $x_j$ from $A_1$ and replace it with $y$.
Repeating this process, we obtain an $\LL$-free subset which contains $A''$ and is at least as large as $A_1$.
\endproof

\subsection{The case of three or more variables}
\label{sec:dec-NP}

The next result shows that for a range of linear equations, $\LL$-\textsc{Free Subset} is $\NP$-complete. For example, the result includes the cases when $\LL$ is $x+y=z$ (i.e. sum-free sets) and $x+y=2z$ (i.e. progression-free sets).
To prove that $\LL$-\textsc{Free Subset} is $\NP$-hard we will use a reduction from the following $\NP$-complete problem~\cite{garey}.
Recall that a \emph{hitting set} $S$ in a hypergraph $H$ is a collection of vertices such that every edge in $H$ contains at least one vertex from $S$.
\begin{framed}
\noindent $\ell$-\textsc{Hitting Set}\newline
\textit{Input:} An $\ell$-uniform hypergraph $H$ and $s \in \mathbb N$.\newline
\textit{Question:} Does $H$ contain a hitting set of size $s$?
\end{framed}

\begin{thm}\label{npthm}
Let $\LL$ be a linear equation of the form $a_1x_1+\dots+a_\ell x_\ell = b y$ where each $a_i \in \mathbb N$ and $b \in \mathbb N$ are fixed and $\ell\geq 2$.
Then $\LL$-\textsc{Free Subset} is $\NP$-complete.
\end{thm}
\proof
Recall from the discussion in Section \ref{sec:complexity} that we can determine in time polynomial in $\size{A}$ whether a set $A$ is $\LL$-free, so $\LL$-\textsc{Free Subset} is in $\NP$.  To show that the problem is $\NP$-complete, we give a reduction from $\ell$-\textsc{Hitting Set}.

Let $(H,s)$ be an instance of $\ell$-\textsc{Hitting Set}.  We construct $A$ and $A''\subseteq A$ as in Lemma~\ref{useful1}  under input $H$ and $\LL$ (taking time polynomial in $\size{A}$).  It suffices to show that $H$ has a hitting set of size $s$ if and only if $A$ contains an $\LL$-free subset of size $k:=|A|-s$.

Observe that $H$ has a hitting set of size $s$ if and only if it has an independent set of size $|H| - s$; by Corollary~\ref{cor:includes-A''}, this holds if and only if $A$ has an $\LL$-free subset of size $|A''| + |H| - s = |A| - s = k$.
\endproof
Note that since Lemma~\ref{useful1} outputs a set $A$ of \emph{natural numbers}, we have actually proved the following stronger result.
\begin{thm}\label{npthm2}
Let $\LL$ be a linear equation of the form $a_1x_1+\dots+a_\ell x_\ell = b y$ where each $a_i \in \mathbb N$ and $b \in \mathbb N$ are fixed and $\ell\geq 2$.  Then $\LL$-\textsc{Free Subset} is $\NP$-complete, even if the input set $A$ is a subset of $\mathbb{N}$.
\end{thm}

\subsection{The two variable case}
\label{sec:dec-2var}

For \emph{any} linear equation $\LL$ in two variables, it is straightforward to see that $\LL$-\textsc{Free Subset} is in $\P$.  Our strategy here is to reduce \emph{to} the problem of finding an independent set (rather than reducing from this problem as in the previous section) and to note that the graph we create must have a very specific structure.

\begin{thm}
Fix any linear equation $\LL$ in two variables.
Then $\LL$-\textsc{Free Subset} is in $\P$.
\end{thm}
\proof
  Let $A = \{a_1,\ldots,a_n\} \subseteq \mathbb{Z}$, and let $k \in \mathbb{N}$.  We now construct $G$ to be the graph with vertex set $A$, where $a_ia_j \in E(G)$ precisely when $(a_i,a_j)$ is a non-trivial solution to $\LL$. 
	Note that we can construct $G$ in  time bounded by a polynomial function of $\size{A}$. The construction of $G$ ensures that a set $A' \subseteq A$ is $\LL$-free if and only if $A'$ is an independent set in $G$.

Notice that a vertex $x$ could lie in a loop. However, in this case $x$ is not adjacent to any other vertex in $G$. All other vertices in $G$ have degree at most $2$. Thus, $G$ is a collection of vertex-disjoint paths, cycles, isolated vertices and loops.
 The largest independent set in both a path and an even cycle on $t$ vertices is $\lceil t/2 \rceil$; the largest independent set in an odd cycle on $t$ vertices is $\lfloor t/2 \rfloor$.
So in time $\mathcal{O}(n)$ we can determine the size of the largest 
independent set in $G$, and thus whether $A$ contains an $\LL$-free subset of size $k$.
\endproof


\section{Approximating the size of the largest $\LL$-\textsc{Free Subset}}\label{sec:ap}

Thus far we have focussed on decision problems involving solution-free sets (``Does the set $A$ contain a solution-free set of a certain size?''), but it is also natural to consider a maximisation problem: ``What is the size of the largest solution-free subset of $A$?''  An efficient algorithm to answer the decision problem can clearly be used to solve the maximisation problem, as we can repeatedly run our decision algorithm with different target sizes; however, as we have demonstrated in Section \ref{sec:decision} that in many cases such an algorithm is unlikely to exist, it makes sense to ask whether we can efficiently approximate the optimisation problem.  We define the maximisation version of $\LL$-\textsc{Free Subset} formally as follows.

\begin{framed}
\noindent \textsc{Maximum $\LL$-Free Subset}\newline
\textit{Input:} A finite set $A \subseteq \mathbb Z$.\newline
\textit{Question:} What is the cardinality of the largest $\LL$-free subset $A' \subseteq A$?
\end{framed}

Given any instance $I$ of an optimisation problem, we denote by $\opt(I)$ the value of the optimal solution to $I$ (so, for example, the cardinality of the largest solution-free subset).  Given a constant $\rho > 1$, we say that an approximation algorithm for the maximisation problem has \emph{performance ratio $\rho$} if, given any instance $I$ of the problem, the algorithm will return a value $x$ such that
$$1 \leq \frac{\opt(I)}{x} \leq \rho.$$
Note that there is a trivial approximation algorithm for \textsc{Maximum Sum-Free Subset} with performance ratio 3: if we always return $|A|/3$ then, as $|A|/3 \leq \opt(A) \leq |A|$, we have $1 \leq \frac{\opt(A)}{|A|/3} \leq 3$ as required.

We might hope to improve on this to obtain, given arbitrary positive $\epsilon$, an approximation algorithm for \textsc{Maximum $\LL$-Free Subset} with performance ratio $1+ \epsilon$.  However, we will show in this section that in certain cases this is no easier than solving the problem exactly.  Specifically, we show that for a large family of 3-variable linear equations (including those defining sum-free and progression-free sets), there is no polynomial-time approximation scheme unless $\P=\NP$.  

A \emph{polynomial-time approximation scheme (PTAS)} for a maximisation problem is an algorithm which, given any instance $I$ of the problem and a constant $\epsilon >0$, returns, in polynomial-time, a value $x$ such that 
$$1 \leq \frac{\opt(I)}{x} \leq 1 + \epsilon.$$
Note that the exponent of the polynomial is allowed to depend on $\epsilon$. 

The complexity class $\APX$ contains all optimisation problems (whose decision version belongs to $\NP$) which can be approximated within some constant factor in polynomial time; this class includes problems which do not admit a PTAS unless $\P=\NP$, so one way to demonstrate that an optimisation problem is unlikely to admit a PTAS is to show that it is hard for the class $\APX$.  In order to show that a problem is $\APX$-hard (and so does not admit a PTAS unless $\P=\NP$), it suffices to give a PTAS reduction from another $\APX$-hard problem.  

\begin{adef}
Let $\Pi_1$ and $\Pi_2$ be maximisation problems.  A \emph{PTAS reduction from $\Pi_1$ to $\Pi_2$} consists of three polynomial-time computable functions $f$, $g$ and $\alpha$ such that:
\begin{enumerate}
\item for any instance $I_1$ of $\Pi_1$ and any constant error parameter $\epsilon$, $f$ produces an instance $I_2 = f(I_1,\epsilon)$ of $\Pi_2$;
\item if $\epsilon > 0$ is any constant and $y$ is any solution to $I_2$ such that 
$\frac{\textrm{opt}(I_2)}{|y|} \leq \alpha(\epsilon)$, then $x = g(I_1,y,\epsilon)$ is a solution to $I_1$ such that $\frac{\textrm{opt}(I_1)}{|x|} \leq 1 + \epsilon$.
\end{enumerate}
\end{adef}
 
We cannot immediately deduce results about the inapproximability of \textsc{Maximum $\LL$-Free Subset} from Corollary~\ref{cor:includes-A''} together the inapproximability of \textsc{Independent Set}, as the cardinality of the largest $\LL$-free subset in $A$ will in general be dominated by the cardinality of $A''$.  However, we can instead reduce from \textsc{Max IS-3}, the problem of finding the size of a maximum independent set in a graph of maximum degree 3, which was shown to be $\APX$-hard by Alimonti and Kann \cite{alimonti00}.  For this reduction we imitate the approach of Froese, Janj, Nichterlein and Niedermeier \cite{froese16}, who obtained a result analogous to Lemma \ref{useful1} when reducing 3-\textsc{Hitting Set} to the problem of finding a maximum subset of points in general position.

Corollary \ref{cor:includes-A''} implies that we have a polynomial-time reduction from \textsc{3-IS} to \textsc{$\LL$-Free Subset} (for suitable $\LL$) in which $|A''| \leq \frac{3|V|}{2}$. 
 Since it is also well-known that in any graph $G$ on $n$ vertices with maximum degree at most $\Delta$, every maximal independent set has cardinality at least $\frac{n}{\Delta + 1}$ (if the independent set is smaller than this, there must be some vertex which does not have any neighbour in the set and so can be added to the independent set), it follows that for every maximal independent set $U$ in $G$ we have $|U| \geq \frac{|V|}{4} \geq \frac{1}{6}|A''|$.

Using this observation, we can now define a PTAS reduction from \textsc{Max IS-3} to \textsc{Maximum $\LL$-Free Subset} for certain 3-variable equations.

\begin{lemma}
Let $\LL$ be a linear equation of the form $a_1x_1 + a_2x_2 = by$, where $a_1,a_2,b \in \mathbb{N}$ are fixed.  Then there is a PTAS reduction from \textsc{Max IS-3} to \textsc{Maximum $\LL$-Free Subset}.
\end{lemma}
\begin{proof}
We define the functions $f$, $g$ and $\alpha$ as follows.  

First, we let $f$ be the function which, given an instance $G$ of \textsc{Max IS-3} (where $G=(V,E)$) and any $\epsilon > 0$, outputs the set $A = \phi_V^{-1}(V) \cup \phi_E^{-1}(E) \subseteq \mathbb{N}$ described in Lemma \ref{useful1}; we know from Lemma \ref{useful1} that we can construct this set in polynomial time.

Next suppose that $B$ is an $\LL$-free subset in $A$.  We can construct in polynomial time a set $\widetilde{B}$, with $|\widetilde{B}| \geq |B|$, such that
\begin{enumerate}
\item $\phi_E^{-1}(E) \subseteq \widetilde{B}$; and
\item $\widetilde{B}$ is a maximal $\LL$-free subset of $A$.
\end{enumerate}
If $B$ fails to satisfy the first condition, we can use the method of Corollary \ref{cor:includes-A''} to obtain a set with this property, and if the resulting set is not maximal we can add elements greedily until this condition is met.  We now define $g$ to be the function which, given an $\LL$-free set $B \subseteq A$ and any $\epsilon > 0$, outputs $\phi_V(\widetilde{B} \setminus \phi_E^{-1}(E))$.

Finally, we define $\alpha$ to be the function $\epsilon \mapsto 1 + \frac{\epsilon}{7}$.  Let us denote by $\opt(G)$ the cardinality of the maximum independent set in $G$, and by $\opt(A)$ the cardinality of the largest $\LL$-free subset in $A$.  Note that $\opt(A) = \opt(G) + |E|$.  To complete the proof, it suffices to demonstrate that, whenever $B$ is an $\LL$-free subset in $A$ such that $\frac{\opt(A)}{|B|} \leq \alpha(\epsilon) = 1 + \frac{\epsilon}{7}$, we have $\frac{\opt(G)}{|I|} \leq 1 + \epsilon$, where $I := \phi_V(\widetilde{B} \setminus \phi_E^{-1}(E))$.  Observe that
\begin{align*}
\frac{\opt(A)}{|B|} &\leq 1 + \frac{\epsilon}{7} \\
\Rightarrow \frac{\opt(A)}{|\widetilde{B}|} &\leq 1 + \frac{\epsilon}{7} \\
\Rightarrow \frac{|E| + \opt(G)}{|E| + |I|} & \leq 1 + \frac{\epsilon}{7} \\
\Rightarrow \frac{|E| + \opt(G)}{|I|} & \leq \left(1 + \frac{\epsilon}{7}\right)\left(\frac{|E| + |I|}{|I|}\right) \\
\Rightarrow \frac{\opt(G)}{|I|} & \leq \left(1 + \frac{\epsilon}{7}\right)\frac{|E|}{|I|} + \left(1 + \frac{\epsilon}{7}\right) - \frac{|E|}{|I|} \\
& = \frac{\epsilon}{7}\frac{|E|}{|I|} + 1 + \frac{\epsilon}{7}.
\end{align*}
Since we know that $|E| \leq \frac{3|V|}{2}$ and, by our assumptions on maximality of $\widetilde{B}$ and hence $I$, we also know that $|I| \geq \frac{|V|}{4}$, it follows that $\frac{|E|}{|I|} \leq 6$.  We can therefore conclude that
$$\frac{\opt(A)}{|B|} \leq 1 + \frac{\epsilon}{7} 
\Rightarrow \frac{\opt(G)}{|I|} \leq 6\frac{\epsilon}{7} + 1 + \frac{\epsilon}{7} = 1 + \epsilon,$$
as required.
\end{proof}

We now obtain our main inapproximability result as an immediate corollary.

\begin{thm}
Let $\LL$ be a linear equation of the form $a_1x_1 + a_2x_2 = by$, where $a_1,a_2,b \in \mathbb{N}$ are fixed.  Then \textsc{Maximum $\LL$-Free Subset} is $\APX$-hard.
\end{thm}

\section{$\LL$-free subsets of arbitrary sets of integers}
\label{sec:arbitrary-sets}

In much of the rest of this paper, we prove complexity results which hold whenever we can guarantee that our input set $A$ will contain a reasonably large $\LL$-free subset.  We already know that this is the case for sum-free subsets (in which case an arbitrary input set $A$ of non-zero integers must contain a sum-free subset of size at least $(|A|+1)/3$); in this section we extend this result to a much larger family of linear equations, proving that, given any homogeneous non-translation-invariant linear equation $\LL$, every finite set of non-zero integers contains an $\LL$-free set of linear size.  Note that a homogeneous linear equation $\LL$ is non-translation-invariant if and only if it can be written in the form $a_1x_1+\dots+a_{k}x_k=b_1y_1+\dots+b_{\ell} y_{\ell}$ for some fixed $a_i,b_i \in \mathbb N$ where $a_1+\dots +a_k \neq b_1+\dots+b_{\ell}$.

For this we use the following simple observation. 
\begin{obs}\label{obs1}
Consider a homogeneous linear equation $\LL$ of the form $a_1x_1+\dots+a_{k}x_k=b_1y_1+\dots+b_{\ell} y_{\ell}$ for some fixed $a_i,b_i \in \mathbb N$ where $a_1+\dots +a_k > b_1+\dots+b_{\ell}$.
Then the interval 
$$I:=\left [\left \lfloor \frac{(b_1+\dots + b_{\ell})n}{a_1+\dots +a_k}\right \rfloor+1, n  \right ]$$
is $\LL$-free. 
\end{obs}
Note that Observation~\ref{obs1} is immediate since $(a_1+\dots +a_k) \min (I) >(b_1+\dots+b_{\ell}) \max(I)$.
\begin{thm}\label{lamthm}
Consider a non-translation-invariant homogeneous linear equation $\LL$.
There exists some $\lambda=\lambda(\LL)>0$ such that, if $n \in \mathbb N$ is sufficiently large, then any set $Z\subseteq \mathbb Z\setminus \{0\}$ so that $|Z|=n$ contains
an $\LL$-free subset of size more than $\lambda n$.
\end{thm}
\proof
Suppose that $\LL$ is of the form $a_1x_1+\dots+a_{k}x_k=b_1y_1+\dots+b_{\ell} y_{\ell}$ for some fixed $a_i,b_i \in \mathbb N$ where $a_1+\dots +a_k \neq b_1+\dots+b_{\ell}$.  Observation~\ref{obs1} implies that there is some $\lambda'=\lambda'(\LL)>0$ such that, if $m' \in \mathbb N$ is sufficiently large, then $[m']$ contains an $\LL$-free subset
of size at least $\lambda 'm'$. We say a subset $S$ of a group $G$ is \emph{$\LL$-free} if $S$ contains \emph{no} solutions to $\LL$.
Set $c:=\max\{(a_1+\dots+a_k), (b_1+\dots +b_{\ell})\}$.

\begin{claim}\label{c11}
There is some  $\lambda:=\lambda'/(2c)>0$ such that, if $m \in \mathbb N$ is sufficiently large, then $\mathbb Z_m$ contains an $\LL$-free subset
of size at least $\lambda m$ that does not contain the zero element.
\end{claim}
\noindent To prove the claim, suppose that $m\in \mathbb N$ is sufficiently large and define $m':=\lfloor m/c\rfloor$. 
Note  a set $S\subseteq \{1,\dots, m'\}$ is an $\LL$-free subset of $\mathbb Z_m$ if and only if $S$ is an $\LL$-free subset of $[m']$.
Indeed, suppose for a contradiction there is a solution $(x_1,\dots, x_k,y_1,\dots, y_{\ell})$ to $\LL$ in $\{1,\dots, m'\}\subseteq \mathbb Z_m$ that is not a solution to $\LL$ when viewed as a subset of $[m']$.
So viewing $x_1,\dots, x_k,y_1,\dots, y_{\ell}$ as integers we have that $a_1x_1+\dots+a_{k}x_k\not = b_1y_1+\dots+b_{\ell} y_{\ell}$ however, $a_1x_1+\dots+a_{k}x_k\equiv b_1y_1+\dots+b_{\ell} y_{\ell}\mod m$. Thus the difference between $a_1x_1+\dots+a_{k}x_k$ and $b_1y_1+\dots+b_{\ell} y_{\ell}$ is at least $m$. This yields a contradiction since, by the definition of $m'$, neither of these numbers is bigger than $m$.

Thus, as $[m']$ contains an $\LL$-free subset of size at least $\lambda 'm'\geq \lambda m$, $\mathbb Z_m$ contains an $\LL$-free subset of size at least $\lambda m$ avoiding the zero element, proving the claim.

\smallskip

The rest of the proof modifies the argument presented in~\cite[Theorem 1.4.1]{alonspencer} that shows every set of $n$ non-zero integers contains a sum-free subset of size more than $n/3$.
Let $n \in \mathbb N$ be sufficiently large and consider any set $Z=\{z_1,\dots, z_n\}$ of non-zero integers. Let $p$ be a prime so that $p>2\max(Z)$. 
Since $n$ is sufficiently large and $p>n$, by the claim we have that $\mathbb Z_p$ contains an $\LL$-free subset $S$ so that $|S|\geq \lambda  p$ and additionally $0\not \in S$.

Choose an integer $x$ uniformly at random from $\{1,2,\dots, p-1\}$, and define $d_1,\dots, d_n$ by $d_i \equiv x z_i \mod p$ where $0\leq d_i <p$.
For every fixed $1\leq i \leq n$, as $x$ ranges over all numbers $1,2,\dots, p-1$, then $d_i$ ranges over all non-zero elements of $\mathbb Z_p$. Therefore,
$\mathbb P(d_i \in S)=|S|/(p-1)>\lambda  $.
So the expected number of elements $z_i$ such that $d_i \in S$ is more than $\lambda n$.
Thus, there is some choice of $x$ with $1\leq x<p$ and a subset $Z'\subseteq Z$ of size $|Z'|>\lambda n$ such that $xz_i$ (mod $p$) $\in S$ for all $z_i \in Z'$.
Since $S$ is $\LL$-free in $\mathbb Z_p$, and $\LL$ is homogeneous, this implies $Z'$ is an $\LL$-free set of integers, as desired. 
\endproof

In the case when $\LL$ is translation-invariant, Ruzsa~\cite{ruzsa} observed that the largest $\LL$-free subset of $[n]$ has size $o(n)$.  So one cannot prove an analogue of Theorem~\ref{lamthm} for such equations $\LL$.

The next result follows immediately from  Theorem~\ref{lamthm}.

\begin{thm}\label{lamthm2}
Consider a non-translation-invariant homogeneous linear equation $\LL$.
There exists some $\lambda=\lambda(\LL)>0$ such that every finite set $Z\subseteq \mathbb Z\setminus \{0\}$  contains
an $\LL$-free subset of size more than $\lambda |Z|$.
\end{thm}
Note that is necessary to restrict our attention to sets $Z\subseteq \mathbb Z\setminus \{0\}$ here since $Z:=\{0\}$ does not contain a non-empty $\LL$-free subset.

It is natural to ask how large $\lambda(\LL)$ can be in the previous theorem.  Let $\mathcal C  (\LL)$ denote the set of all positive reals $\kappa$ so that Theorem~\ref{lamthm2} holds with $\kappa$ playing the role of $\lambda$, and define $\mathcal C '  (\LL)$ analogously now with respect to Theorem~\ref{lamthm}. We claim that $\mathcal C  (\LL)=\mathcal C '  (\LL)$. It is immediate that $\mathcal C  (\LL)\subseteq \mathcal C '  (\LL)$. To see that there is no $\lambda \in \mathcal C '  (\LL)
\setminus \mathcal C  (\LL)$ consider the following observation:
Suppose $Z\subseteq \mathbb Z\setminus \{0\}$ is such that it does not contain an $\LL$-free subset of size more than $\lambda |Z|$ for some $\lambda >0$. Set $z:=|Z|$. Then for every $n \in\mathbb N$ there is a set $Z'\subseteq \mathbb Z\setminus\{0\}$ of size $zn$
such that it does not contain an $\LL$-free subset of size more than $\lambda |Z'|$. Indeed, writing $cZ$ as shorthand for $\{cz: z \in Z\}$, one can choose $Z'$ to be the union of $c_1Z, \dots  ,c_n Z$ where the $c_i$s are positive integers chosen to ensure the sets $c_iZ$ are pairwise disjoint. (Notice we required that $0 \not \in  Z$ to ensure this.)

Define
$$\kappa (\LL):=\sup (\mathcal C(\LL)).$$

Write $\LL$ as $a_1x_1+\dots+a_{k}x_k=b_1y_1+\dots+b_{\ell} y_{\ell}$ for some fixed $a_i,b_i \in \mathbb N$ where $a_1+\dots +a_k > b_1+\dots+b_{\ell}$.
We remark that it is easy to check  in the statement of Theorem~\ref{lamthm}, and therefore Theorem~\ref{lamthm2}, one can set 
\begin{align}\label{quote}
\lambda (\LL)=\frac{1}{2c}\left (1-\frac{b_1+\dots+b_\ell}{a_1+\dots+a_k}\right )
\end{align}
where here we define  $c:=\max\{(a_1+\dots+a_k), (b_1+\dots +b_{\ell})\}$.
That is, $\kappa (\LL)\geq \frac{1}{2c}\left (1-\frac{b_1+\dots+b_\ell}{a_1+\dots+a_k}\right )$.

In the case when $\LL$ is $x+y=z$ we know that $\kappa(\LL)=1/3$ and this supremum is attained. Indeed, recall that every set of $n$ non-zero integers has a sum-free subset of size at least $(n+1)/3$ \cite{alonkleitman} whilst  there are sets of positive integers $A$ of size $n$ such that $A$ does not contain any sum-free subset of size greater than $n/3+o(n)$ \cite{egm}.
It would be interesting to determine $\kappa (\LL)$ for other equations $\LL$.
\begin{problem}
Determine $\kappa (\LL)$ for  non-translation-invariant homogeneous linear equations $\LL$.
\end{problem}

\section{Parameterised complexity of the decision problem}
\label{sec:param-dec}

In this section we consider the complexity of $\LL$-\textsc{Free Subset} with respect to two natural parameterisations, namely the number of elements in the sum-free subset ($k$) and the number of elements \emph{not} in this subset ($|A| - k$).

First, it is straightforward, using the results of Section \ref{sec:arbitrary-sets}, to see that the problem is in $\FPT$ when parameterised by $k$, whenever $\LL$ satisfies the conditions of Theorem \ref{lamthm2}.

\begin{prop}
Let $\LL$ be a non-translation-invariant homogeneous linear equation.  Then $\LL$-\textsc{Free Subset}, parameterised by $k$, is in $\FPT$.
\end{prop}
\begin{proof}
By Theorem \ref{lamthm2} and equation (\ref{quote}), we know that there exists an explicit constant $\lambda = \lambda(\LL) > 0$ such that any finite set $A \subset \mathbb{Z}\setminus \{0\}$  contains an $\LL$-free subset of size at least $\lambda|A|$.
Let $(A,k)$ be the input to  $\LL$-\textsc{Free Subset}.  We first set $A' := A \setminus \{0\}$: note that $|A'| \geq |A| - 1$, and $A'$ contains an $\LL$-free subset of size $k$ if and only if $A$ does, as no $\LL$-free subset of $A$ can contain $0$.  There are two cases:
\begin{enumerate}
\item $|A'| \leq \frac{k}{\lambda}$: in this case we can solve the problem by means of a brute-force search in time bounded by a function of $k$ only;
\item $|A'| > \frac{k}{\lambda}$: in this case we can immediately return the answer YES by the choice of $\lambda$.
\end{enumerate}
We therefore obtain an fpt-algorithm with respect to the parameter $k$ for $\LL$-\textsc{Free Subset}, by first considering the cardinality of $A'$ to determine which of the two cases is relevant, and then applying the appropriate method.
\end{proof}

We now argue that the problem is also in $\FPT$ with respect to the dual parameterisation.  We have seen that, subject to certain conditions on $\LL$, we can reduce an appropriate version of \textsc{Hitting Set} to $\LL$-\textsc{Free Subset}; we now show that we can also reduce in the opposite direction, from $\LL$-\textsc{Free Subset} to an appropriate (different) version of \textsc{Hitting Set}. 

\begin{lemma}
\label{lma:to-HS}
Let $\LL$ be any fixed linear equation with $\ell$ variables, and let $A \subseteq \mathbb{Z}$ be finite.  Then we can construct, in time polynomial in $\size{A}$, a hypergraph $G$ on $|A|$ vertices in which every edge contains at most $\ell$ vertices, such that there is a one-to-one correspondence between $\LL$-free subsets of $A$ of cardinality $k$ and hitting sets in $G$ of cardinality $|A|-k$.
\end{lemma}
\begin{proof}
Suppose without loss of generality that $\LL$ is of the form $a_1x_1 + \cdots + a_{\ell}x_{\ell} = b$, where $a_1,\ldots,a_{\ell},b \in \mathbb{Z}$. 
Let $G$ be the hypergraph with vertex set $A$ and edge set
$$E := \left\lbrace \{x_1,\ldots,x_{\ell}\}: \text{ $(x_1,\dots,x_{\ell})$ is a non-trivial solution to $\LL$} \right\rbrace.$$
Note that $x_1,\ldots,x_{\ell}$ are not necessarily all distinct, so while every edge in $E$ contains \emph{at most} $\ell$ vertices, an edge may contain strictly fewer than $\ell$ vertices.  It is clear that we can construct $G$ in time $\mathcal{O}(\size{A}^{\ell})$.

We claim that the function $\theta:A \rightarrow A$ defined by $\theta(B) = A \setminus B$ is a bijection from $\LL$-free subsets of $A$ to hitting sets of $G$.  It is clear that this function is a bijection (indeed it is self-inverse), and moreover we have that if $|B| = k$ then $|\theta(B)| = |A|-k$; in order to complete the proof it remains only to show that if $B$ is an $\LL$-free subset of $A$ then $\theta(B)$ is a hitting set of $G$, and that if $B$ is a hitting set of $G$ then $\theta^{-1}(B)=\theta(B)$ is an $\LL$-free subset of $A$.

Suppose that $B$ is an $\LL$-free subset of $A$.  Then, by definition of $E$, there is no $e \in E$ such that $e \subseteq A$.  It follows immediately that every edge $e \in E$ contains at least one vertex of $A \setminus B = \theta(B)$, so $\theta(B)$ is a hitting set in $G$.  Conversely, suppose that $B$ is a hitting set in $G$.  Then we know that there is no $e \in E$ such that $e \subseteq A \setminus B$.  It follows from the definition of $E$ that $A \setminus B = \theta(B)$ is $\LL$-free.
\end{proof}

In particular, this result means that, in order to decide if $A$ contains a solution-free subset of cardinality $k$, it suffices to determine whether a hypergraph on $|A|$ vertices contains a hitting set of cardinality $|A|-k$.  We can therefore make use of known algorithms for the following parameterised problem.

\begin{framed}
\noindent \textbf{p}-card-\textsc{Hitting Set}\newline
\textit{Input:} A hypergraph $G=(V,E)$ and $s \in \mathbb{N}$.\\
\textit{Parameter:} $s + d$, where $d = \max_{e \in E} |e|$. \newline
\textit{Question:} Does $G$ contain a hitting set of cardinality $s$?
\end{framed}

This problem is known to belong to $\FPT$ \cite[Theorem 1.14]{flumgrohe}, so we obtain the following result as an immediate corollary.

\begin{thm}\label{thm:dual-param-dec}
Let $\LL$ be any fixed linear equation.  Then $\LL$-\textsc{Free Subset}, parameterised by $|A|-k$, belongs to $\FPT$.
\end{thm}


\section{$\LL$-free subsets covering a given fraction of elements}\label{sec:6}

We know, by Theorem~\ref{npthm}, that there is unlikely to be a polynomial time algorithm to decide whether a set $A$ has an $\LL$-free subset of size $k$, for arbitrary $k\in \mathbb N$.  It is therefore natural to ask whether we can efficiently solve a restricted version of the problem in which we want to determine whether a finite set $A$ of (non-zero) integers contains an $\LL$-free subset 
that houses some fixed proportion of the elements of $A$. Given any linear equation $\LL$  and
$0<\eps<1$, we define the following problem.
\begin{framed}
\noindent $\eps$-$\LL$-\textsc{Free Subset}\newline
\textit{Input:} A finite set $A \subseteq \mathbb Z \setminus \{0\}$.\newline
\textit{Question:} Does there exist an $\LL$-free subset $A' \subseteq A$ such that $|A'|\geq \eps|A|$?
\end{framed}
In the case when $\LL$ is $x+y=z$ we refer to $\eps$-$\LL$-\textsc{Free Subset} as $\eps$-\textsc{Sum-Free Subset}. 
Note that $\eps$-$\LL$-\textsc{Free Subset} concerns finite sets of non-zero integers $A$; thus, the definition of $\kappa (\LL)$ (given in Section~\ref{sec:arbitrary-sets}) immediately implies that $\eps$-$\LL$-\textsc{Free Subset} is in
$\P$ for all $\eps \leq \kappa (\LL)$, as in this case every instance is a yes-instance.

Further, recall from Section~\ref{sec:complexity} that, given any fixed linear equation $\LL$, we can decide in time polynomial in $\size{A'}$ whether a set $A' \subseteq \mathbb{Z}$ is $\LL$-free, so $\epsilon$-$\LL$-\textsc{Free Subset} clearly belongs to $\NP$.

We will show in Section \ref{62} that, for certain choices of $\LL$ and $\epsilon$, the $\epsilon$-$\LL$-\textsc{Free Subset} problem is no easier than $\LL$-\textsc{Free Subset}.  
For this, we will actually restrict our attention to the case when we have input set $A\subseteq \mathbb N$.
In this case, we need to be able to add elements to  $A $ without creating any additional solutions; we prove results about this in Section \ref{sec:extend}.

\subsection{Extending sets without creating additional solutions} \label{sec:extend}
In Section~\ref{62}, and also later in Section~\ref{sec:counting}, we will make use of the following lemma, which allows us to extend sets without creating additional solutions to an equation.

\begin{lemma}\label{extend1}
Suppose $\LL$ is a linear equation $ax+by=cz$ where $a,b,c \in \mathbb N$ are fixed and $a+b \not = c$. 
Suppose $A\subseteq \mathbb N$ is a finite set and $t \in \mathbb N$ so that $t>|A|$. 
Then there is a set $B\subseteq \mathbb N$ such that:
\begin{itemize}
\item[(i)] $|B|=t$;
\item[(ii)] $A\subseteq B$;
\item[(iii)] the only solutions to $\LL$ in $B$ lie in $A$;
\item[(iv)] $\max (B)= \mathcal{O}\left(t\left(\max(A)\right)^2\right)$.
\end{itemize}
Moreover, $B$ can be computed in time polynomial in $\size{A}$ and $t$.
\end{lemma}
\proof
Write $m:=\max (A)$ and set $\tau:=\min\{\frac{c}{a+b}, \frac{a+b}{c}\}$. Define $N\in \mathbb N$ to be the smallest natural number so that $\lfloor \tau N   \rfloor \geq 2(a+b+c)m$ and $N-\lfloor \tau N \rfloor \geq t$.
This choice of $N$ means that  $\lfloor \tau (N-1)   \rfloor < 2(a+b+c)m$ or $(N-1)-\lfloor \tau (N-1) \rfloor< t$.
Thus,
\begin{align}\label{ll}
N\leq \max \left\{\frac{t}{1-\tau}, \frac{2(a+b+c)m}{\tau}\right \}+1.
\end{align}

By Observation~\ref{obs1} and the choice of $N$, $[N]$ contains an $\LL$-free subset $I'$ so that
$$I'\subseteq [\lfloor \tau N   \rfloor+1, N]$$
and $|I'|=t-|A|$.

Chebyshev's theorem implies that there is a prime $p$ so that $abcm<p<2abcm$. Set $I'':=pI'$ and let $B:=A\cup I''$.  Note that we can clearly determine $N$ and hence construct $I'$ in time bounded by a polynomial function of $\size{A}$ and $t$.  We can determine an appropriate value for $p$ (and then construct $I''$) by exhaustively searching the specified interval and testing for primality in polynomial time (using the AKS test \cite{aks}).  This set immediately satisfies (ii) and, since $p>m$, $A$ and $I''$ are disjoint so (i) is satisfied.  Moreover, it follows from \eqref{ll} and the choice of $p$ that $\max(B) \leq \max\left\lbrace \frac{2abctm}{1 - \tau}, \frac{4m^2abc(a+b+c)}{\tau}\right\rbrace + 2abcm = \mathcal{O}\left(t \left(\max(A)\right)^2\right)$, so (iv) is satisfied. 

To see that (iii) is satisfied, we first observe that there are no solutions to $\LL$ in $I''$. Since $\min(I'') > 2(a+b+c)m$ it is easy to check that there are no solutions to $\LL$ in $B$ which consist of two elements from $A$ and one element from $I''$.
Suppose there is a solution to $\LL$ in $B$ which consists of two elements $z_1,z_2$ from $I''$ and one element $z_3$ from $A$. 
Consider the case when $az_1+bz_2=cz_3$ (the other cases follow identically).
Since every element of $I''$ is divisible by $p$ we have that $p$ divides $cz_3$. So as $c<p$ this implies $p$ must divide $z_3$. However, no element of $A$ is divisible by $p$ since $\max (A)=m<p$, a contradiction.  Hence $B$ satisfies condition (iii).  This completes the proof.
\endproof

We can also prove an analogous result for equations $\LL$ of the form $ax+by=cz$ where $a,b,c\in \mathbb N$ are fixed and $a+b=c$.  To do so, we will need the following fact.

\begin{fact}\label{easy}
Suppose $\LL$ is a linear equation $ax+by=cz$ where $a,b,c \in \mathbb N$ are fixed; $a+b  = c$; and $a\leq b$. 
Given any $x_1<x_2<x_3$ that form a solution $(x,y,z)$ to $\LL$ in $\mathbb N$, we have that $x_2$ plays the role of $z$ and $cx_2>ax_3$.
\end{fact}
\proof
 Consider any $x_1<x_2<x_3$ that form a solution $(x,y,z)$ to $\LL$ in $\mathbb N$. 
Note that since $a+b=c$, we have $cx_3>\max\{(ax_1+bx_2), (ax_2+bx_1)\}$. Thus $x_3$ cannot play the role of $z$.
Further, $x_2$ must play the role of $z$. Indeed, otherwise $x_1$  plays the role of $z$ and then we have $ax+by>cz$, a contradiction.
Altogether this implies that $cx_2>ax_3$.
\endproof

\begin{lemma}\label{extend2}
Suppose $\LL$ is a linear equation $ax+by=cz$ where $a,b,c \in \mathbb N$ are fixed and $a+b  = c$. 
Suppose $A\subseteq \mathbb N$ is a finite set and $t \in \mathbb N$. Then there is a set $B\subseteq \mathbb N$ such that:
\begin{itemize}
\item[(i)] $|B|=|A|+t$;
\item[(ii)] $A\subseteq B$;
\item[(iii)] the only non-trivial solutions to $\LL$ in $B$ lie in $A$;
\item[(iv)] $\max (B)=2\left(\max(A)\right)c^t$.
\end{itemize}
Moreover, $B$ can be computed in time polynomial in $t$ and $\size{A}$.
\end{lemma}
\proof
Without loss of generality assume that $a\leq b$.
Note that  $(x,y,z)$ is a non-trivial solution to $\LL$ if and only if $(x,y,z)$ is a solution to $\LL$ with $x,y,z$ distinct.

Set $m := \max(A)$, and define $A':=\{ c^i\cdot 2m \ : \ i \in [t] \}$; we can clearly construct $B$ in time polynomial in $t$ and $\size{A}$. We claim that $B:=A\cup A'$ is our desired set. Certainly (i), (ii) and (iv)  follow immediately.

We now prove (iii). Suppose $x_1<x_2<x_3$  form a solution $(x,y,z)$ to $\LL$ in $A'$. Then by Fact~\ref{easy} we must have that
$cx_2>ax_3$. However, by definition of $A'$, $ax_3\geq x_3\geq c  x_2$, a contradiction. So $A'$ does not contain any non-trivial solutions to $\LL$.
The same argument shows that there are no non-trivial solutions to $\LL$ in $B$ which contain two elements from $A'$ and one element from $A$.
Finally suppose  $x_1<x_2<x_3$  form a solution $(x,y,z)$ to $\LL$ in $B$ where $x_1,x_2 \in A$ and $x_3 \in A'$. As before we must have that
$cx_2>ax_3$. However, $ax_3>acm\geq cx_2$ by definition of $A'$, a contradiction. This proves (iii).
\endproof

Lemma~\ref{extend2} will be applied in the next subsection to prove that for any equation $\LL$ as in its statement, $\eps$-$\LL$-\textsc{Free Subset} is $\NP$-complete for any $0<\eps <1$.

\subsection{Hardness of $\eps$-$\LL$-\textsc{Free Subset}}\label{62}

In this section we show that, in two specific cases, $\eps$-$\LL$-\textsc{Free Subset} is $\NP$-complete.  We begin with the case of sum-free subsets.  Note that if $\eps \leq \frac{1}{3}$ then the problem is trivially in $\P$ as the answer is always ``yes''; also if $\eps = 1$ then it suffices to check whether the input set is sum-free (which can be done in polynomial time).  We now demonstrate that the problem is $\NP$-complete for all other values of $\eps$.  Recall that, given a set $X \subseteq \mathbb{N}$ and $y \in \mathbb{N}$, we write $yX$ as shorthand for $\{yx: x \in X\}$.

\begin{thm}\label{thmc}
Given any rational $1/3<\eps<1$, $\eps$-\textsc{Sum-Free Subset} is $\NP$-complete.
\end{thm}

\proof
Recall that $\eps$-\textsc{Sum-Free Subset} belongs to $\NP$.  
To show that the problem is $\NP$-hard, we 
 describe a reduction from \textsc{Sum-Free Subset} (restricted to inputs $A\subseteq \mathbb N$), shown to be $\NP$-hard in Theorem~\ref{npthm2}.

Suppose that $(A,k)$ is an instance of \textsc{Sum-Free Subset} where $A\subseteq \mathbb N$. 
We will define a set $B \subseteq \mathbb N$ such that $B$  has a sum-free subset of size at least $\eps|B|$ if and only if $A$ has a sum-free subset of size $k$.
The construction of $B$ depends on the value $k$.

First suppose that $k\leq \eps|A|$.  Set
$$d := \left\lceil \frac{\eps |A| - k}{1 - \eps} \right \rceil,$$
so that $d$ is the least positive integer such that $\eps (|A| + d) \leq k+d$ and hence $\lceil \eps(|A| + d)\rceil = k + d$.  Note that $d = \mathcal{O}(|A|)$.  By Lemma~\ref{extend1} we can construct, in time polynomial in $\size{A}$ and $d$, a set $B \subseteq \mathbb{N}$ of size $|A|+d$ such that $B$ has a sum-free subset of size $k+d$ if and only if $A$ has a sum-free subset of size $k$.  By the choice of $d$, we know that $B$ has a sum-free subset of size $k+d$ if and only if $B$ has a sum-free subset of size at least $\eps(|A| + d)$; so $B$ is a yes-instance for $\eps$-\textsc{Sum-Free Subset} if and only if $(A,k)$ is a yes-instance for \textsc{Sum-Free Subset}.

Now suppose $k >\eps|A|$.
The result of Eberhard, Green and Manners~\cite{egm} implies that there is a set $S\subseteq \mathbb N$ such that the largest sum-free subset of $S$ has size precisely $\eps'|S|$ where $1/3<\eps'<\eps$. 
In particular, through an exhaustive search, one can construct such a set $S$.  Crucially, $S$ is independent of our input $(A,k)$ (so $\size{S}$, $|S|$ and $\max (S)$  are all fixed constants).

Set 
$$r := \left\lceil \frac{k - \eps|A|}{(\eps - \eps')|S|} \right\rceil,$$
and note that $r = \mathcal{O}(|A|)$.  Set $m := \max(A)$, $m' := \max(S)$, and define $d_i := 3^im(m')^{i-1}$ for each $1 \leq i \leq r$.  Note that $\log d_r = \mathcal{O}(r + \log m)$, so for each $1 \leq i \leq r$, we can represent the set $d_iS$ in space $\mathcal{O}(\size{A})$.  Now define
$$A^*:= A \cup \bigcup _{1\leq i \leq r} d_i S.$$
The choice of the $d_i$ ensures the only solutions to $x+y=z$ in $A^*$ are such that $x,y,z\in A$ or $x,y,z\in d_i S$ for some $i \in [r]$. The largest sum-free set in  $d_iS$ is of size $\eps'|S|=\eps'|d_iS|$.
Define $k^*:=k+r\eps'|S|$, and observe that $A$ has a sum-free subset of size $k$ if and only if $A^*$ has a sum-free subset of size $k^*$.
By definition of $r$, $r \eps' |S| \leq r \eps |S| + \eps |A| - k$, so we see that $k^*\leq \eps(|A|+r|S|)=\eps|A^*|$. Now we can argue precisely as in the first case: from $A^*$ one can construct a set $B$ in time polynomial in $\size{A}$ so that $A^*$ has a sum-free subset of size $k^*$ if and only if $B$ has a sum-free subset of size at least $\eps|B|$.
In particular, $B$ will be a yes-instance for $\eps$-\textsc{Sum-Free Subset} if and only if $(A,k)$ is a yes-instance for \textsc{Sum-Free Subset}, as required.
\endproof

We are also able to prove an $\NP$-completeness result in the only other cases of three-variable equations $\LL$ where $\kappa(\LL)$ is known, using a slight variation on the method of Theorem \ref{thmc}.  
In particular, the following result covers the case of progression-free sets.
Recall that if $\LL$ is translation-invariant, Ruzsa~\cite{ruzsa} observed that the largest $\LL$-free subset of $[n]$ has size $o(n)$ (and so $\kappa (\LL)=0$).

\begin{thm}\label{thmcti}
Consider any rational $0 <\eps<1$ and let $\LL$ denote the equation $ax+by=cz$ where $a,b,c\in \mathbb N$ and $a+b=c$. Then $\eps$-$\LL$-\textsc{Free Subset} is $\NP$-complete.
\end{thm}
\proof
Fix $\eps$ and $\LL$ as in the statement of the theorem; we will assume without loss of generality that $b \geq a$. Recall that $\eps$-$\LL$-\textsc{Free Subset} is in $\NP$.  To show $\NP$-hardness, we once again give a reduction from $\LL$-\textsc{Free Subset}
(restricted to inputs $A\subseteq \mathbb N$), shown to be $\NP$-hard in Theorem~\ref{npthm2}.

Suppose that $(A,k)$ is an instance of $\LL$-\textsc{Free Subset}, where $A \subseteq \mathbb{N}$.
We will define a set $B \subseteq \mathbb N$ such that $B$  has an $\LL$-free subset of size at least $\eps|B|$ if and only if $A$ has an $\LL$-free subset of size $k$.
The construction of $B$ depends on the value $k$.

First suppose that $k\leq \eps|A|$. 
As in the proof of Theorem~\ref{thmc}, we define $d$ so that $\lceil \eps(|A|+d)\rceil=k+d$.  By Lemma \ref{extend2} we can construct, in time bounded by a polynomial function of $\size{A}$, a set $B \subseteq \mathbb N$ of size $|A|+d$ such that (by conditions (i)--(iii) of the lemma) $B$ has an $\LL$-free subset of size $k+d$ if and only if $A$ has an $\LL$-free subset of size $k$.  By our choice of $d$, this means that $B$ is a yes-instance to $\eps$-$\LL$-\textsc{Free Subset} if and only if $(A,k)$ is a yes-instance to $\LL$-\textsc{Free Subset}.

Now suppose $k >\eps|A|$. Since the largest $\LL$-free subset of $[n]$ has size $o(n)$, 
we can find by exhaustive search a set $S\subseteq \mathbb N$ such that the largest $\LL$-free subset of $S$ has size $\eps'|S|$ for some $0<\eps'<\eps$. 
Crucially, $S$ is independent of our input $(A,k)$ (so $|S|$, $\max (S)$  and $\size{S}$ are all fixed constants).

As in the proof of Theorem \ref{thmc}, we set 
$$r := \left \lceil \frac{k - \eps |A|}{(\eps - \eps')|S|} \right \rceil.$$
Set $m := \max (A)$, $m' := \max(S)$, and define $d_i := 3^ic^im(m')^{i-1}$ for each $1 \leq i \leq r$.  Note that, for each $d_i$, we can represent the set $d_iS$ in space $\mathcal{O}(\size{A})$.  Observe also that $\max(A) < \min (d_1S)$ and, for $1 \leq i \leq r-1$, $\max(d_iS) < \min(d_{i+1}S)$.

Now set 
$$A^*:= A \cup \bigcup _{1\leq i \leq r} d_i S.$$
We claim that the only solutions to $ax+by=cz$ in $A^*$ are such that $x,y,z\in A$ or $x,y,z\in d_i S$ for some $i \in [r]$.
To see this, first suppose there are $x_1<x_2<x_3$ in $A^*$ so that $x_3\in d_i S$ for some $i \in [r]$,  $x_2 \not \in d_i S$, and $x_1,x_2,x_3$ form a solution to $\LL$. Then by Fact~\ref{easy} we have that
$cx_2 >ax_3$. However, we also know that, if $i > 1$, $ax_3\geq x_3 \geq d_i = 3cm'd_{i-1} >cm'd_{i-1} \geq cx_2$; if $i=1$ then $ax_3\geq 3acm>cm \geq cx_2$.  In either case this gives a contradiction.
Next suppose there are $x_1<x_2<x_3$ in $A^*$ so that $x_2,x_3\in d_i S$ for some $i \in [r]$, $x_1 \not \in d_i S$, and $x_1,x_2,x_3$ form a solution to $\LL$.
 Suppose $i>1$. Then $d_i$ divides $x_2$ and $x_3$ and so, by Fact~\ref{easy}, $d_i$ divides $ax_1$ or $bx_1$. In particular, we have that $bx_1\geq d_i$.
However, since  $x_1 \not \in d_i S$, we have that $x_1\leq d_{i-1}m'$ and so $bx_1\leq cm'd_{i-1} < 3cm'd_{i-1} = d_i$, a contradiction. The case $i=1$ yields an analogous contradiction.
Altogether this indeed proves the only solutions $(x,y,z)$ to $\LL$ in $A^*$ are such that $x,y,z\in A$ or $x,y,z\in d_i S$ for some $i \in [r]$.

Now we can continue as in the proof of Theorem~\ref{thmc}.  Note that the largest $\LL$-free set in  $d_iS$ is of size $\eps'|S|=\eps'|d_iS|$, and set $k^*:=k+r\eps'|S|$, so that $A$ has an $\LL$-free subset of size $k$ if and only if $A^*$ has an $\LL$-free subset of size $k^*$.
By definition of $r$, we have $k^*\leq \eps(|A|+r|S|)=\eps|A^*|$, so we can now argue precisely as in the first case to obtain a set $B$ such that $B$ is a yes-instance to $\eps$-$\LL$-\textsc{Free Subset} if and only if $(A,k)$ is a yes-instance to $\LL$-\textsc{Free Subset}.
\endproof

\section{Counting solution-free subsets of a given size}\label{sec:counting}

Consider the following counting problem.

\begin{framed}
\noindent \countprob\\ 
\textit{Input:} A finite set $A \subseteq \mathbb Z$ and $k \in \mathbb N$.\newline
\textit{Question:} How many $\LL$-free subsets of $A$ have cardinality exactly $k$?
\end{framed}

It is clear that, whenever $\LL$ satisfies the conditions of Theorem~\ref{npthm}, there cannot be any polynomial-time algorithm for this problem unless $\P = \NP$, as such an algorithm would certainly tell us whether or not the number of such subsets is zero and hence solve the decision problem.  However, it is interesting to consider the complexity of the counting problem with respect to the parameterisations $k$ and $|A| - k$, as we saw in Section \ref{sec:param-dec} that the decision problem is tractable with respect to both of these parameterisations.

In Section~\ref{sec:counting-dual} we show that the counting problem is in $\FPT$ when parameterised by $|A| - k$; similar to the proof of Theorem \ref{thm:dual-param-dec}, this result relies on a reduction to a counting version of \textsc{Hitting Set}.  In contrast, we show in Section \ref{sec:counting-k} that, for certain equations $\LL$, the counting problem is unlikely to admit an fpt-algorithm when parameterised by $k$; however, in many cases there is an efficient parameterised algorithm to solve this problem \emph{approximately}, as we will see in Section \ref{sec:counting-approx}.  We begin in Section \ref{sec:counting-cplx} with some background on the theory of parameterised counting complexity.

\subsection{Parameterised counting complexity}
\label{sec:counting-cplx}

We make use of the theory of parameterised counting complexity developed by Flum and Grohe \cite{flum04,flumgrohe}.  Let $\Sigma$ be a finite alphabet.  A parameterised counting problem is formally defined to be a pair $(\Pi,\kappa)$ where $\Pi: \Sigma^* \rightarrow \mathbb{N}_0$ is a function  and $\kappa: \Sigma^* \rightarrow \mathbb{N}$ is a parameterisation (a polynomial-time computable mapping).  Flum and Grohe define two types of parameterised counting reductions, \emph{fpt parsimonious reductions} and \emph{fpt Turing reductions}.  The latter is more flexible than the former, as it does not require us to preserve the number of witnesses as we tranform between problems; rather we must be able to compute the number of witnesses in one problem using information about the number of witnesses in one or more instances of the other problem, which allows us to make use of several standard techniques for counting reductions (such as polynomial interpolation and matrix inversion).  
\begin{adef}
An fpt Turing reduction from $(\Pi,\kappa)$ to $(\Pi',\kappa')$ is an algorithm $A$ with an oracle to $\Pi'$ such that
\begin{enumerate}
\item $A$ computes $\Pi$,
\item $A$ is an fpt-algorithm with respect to $\kappa$, and
\item there is a computable function $g:\mathbb{N} \rightarrow \mathbb{N}$ such that for all oracle queries ``$\Pi'(I') = \; ?$'' posed by $A$ on input $I$ we have $\kappa'(I') \leq g\left(\kappa(I)\right)$.
\end{enumerate}
In this case we write $(\Pi,\kappa)$ \leqfptT $(\Pi',\kappa')$.
\end{adef}
There is an analogue of the W-hierarchy for counting problems; in order to demonstrate that a parameterised counting problem is unlikely to belong to $\FPT$ it suffices to show that it is hard (with respect to fpt-Turing reductions) for the first level of this heirarcy, \#W[1] (see \cite{flumgrohe} for the formal definition of the class \#W[1]).

A parameterised counting problem is considered to be efficiently approximable if it admits a \emph{fixed parameter tractable randomised approximation scheme (FPTRAS)}, which is defined as follows:
\begin{adef}
A fixed parameter tractable randomised approximation scheme (FPTRAS) for a parameterised counting problem $(\Pi,\kappa)$ is a randomised approximation scheme that takes an instance $I$ of $\Pi$ (with $|I| = n$), and rational numbers $\eps > 0$ and $0 < \delta < 1$, and in time $f(\kappa(I)) \cdot g(n,1/\eps,\log(1/\delta))$ (where $f$ is any computable function, and $g$ is a polynomial in $n$, $1/\eps$ and $\log(1 / \delta)$) outputs a rational number $z$ such that
$$\mathbb{P}[(1-\eps)\Pi(I) \leq z \leq (1 + \eps)\Pi(I)] \geq 1 - \delta.$$
\end{adef}

\subsection{Parameterisation by the number of elements not included in the solution-free set}
\label{sec:counting-dual}

In this section we show that the counting problem, parameterised by $|A|-k$, is in $\FPT$.  This is a straightforward extension of the argument used for the decision problem in Section \ref{sec:param-dec}: since there is a one-to-one correspondence between $\LL$-free subsets of cardinality $k$ and hitting sets of cardinality $|A|-k$ in the construction described in Lemma \ref{lma:to-HS}, we can also make use of parameterised algorithms for \emph{counting} hitting sets to count $\LL$-free subsets.  Thurley \cite{thurley07} describes an fpt-algorithm for the following counting version of the problem.  

\begin{framed}
\noindent \#\textbf{p}-card-\textsc{Hitting Set}\newline
\textit{Input:} A hypergraph $G=(V,E)$ and $s \in \mathbb{N}$.\\
\textit{Parameter:} $s + d$, where $d = \max_{e \in E} |e|$. \newline
\textit{Question:} How many hitting sets in $G$ have cardinality exactly $s$?
\end{framed}

As for Theorem \ref{thm:dual-param-dec}, our result now follows immediately.

\begin{thm}
Let $\LL$ be any fixed linear equation.  Then \countprob, parameterised by $|A|-k$, belongs to $\FPT$.
\end{thm}

\subsection{Parameterisation by the cardinality of the solution-free set}
\label{sec:counting-k}
 
In contrast with the positive result in the previous section, we now show that there is unlikely to be an fpt-algorithm with respect to the parameter $k$ to solve \countprob.  To do this, we give an fpt-Turing reduction from the following problem, which can easily be shown to be \#W[1]-hard by means of a reduction from \paramcount{Clique} (shown to be \#W[1]-hard in \cite{flum04}), along the same lines as the proof of the W[1]-hardness of \paramdec{Multicolour Clique} in \cite{fellows09}.

\begin{framed}
\noindent \paramcount{Multicolour Clique}\newline
\textit{Input:} A graph $G = (V,E)$, and a partition of $V$ into $k$ sets $V_1,\ldots,V_k$.\newline
\textit{Parameter:} $k$ \newline
\textit{Question:} How many $k$-vertex cliques in $G$ contain exactly one vertex from each set $V_1,\ldots,V_k$?
\end{framed}

\noindent
When reducing from \paramdec{Multicolour Clique} or its counting version, it is standard practice to assume that, for each $1 \leq i < j \leq k$, the number of edges from $V_i$ to $V_j$ is equal.  We can make this assumption without loss of generality because we can easily transform an instance which does not have this property to one which does without changing the number of multicolour cliques; note that if the input does not already satisfy this condition then $k \geq 3$.  We set $q := \max\{e(V_i,V_j): 1 \leq i < j \leq k\}$ (where $e(A,B)$ denotes the number of edges with one endpoint in $A$ and the other in $B$), and for any pair of sets $(V_i,V_j)$ where $e(V_i,V_j) = q'<q$, we add vertices $\{u_1,\ldots,u_{q-q'}\}$ to $V_i$ and $\{w_1,\ldots,w_{q-q'}\}$ to $V_j$, and the set of edges $\{u_rw_r:1 \leq r \leq q-q'\}$; note that the largest cliques created by this process contain two vertices.  

We in fact reduce \paramcount{Multicolour Clique} to a multicolour version of \countprob, defined as follows.

\begin{framed}
\noindent \paramcount{Multicolour $\LL$-Free Subset}\newline
\textit{Input:} A $k$-tuple of disjoint subsets $A_1,\ldots,A_k \subseteq \mathbb{Z}$.\newline
\textit{Parameter:} $k$ \newline
\textit{Question:} How many $\LL$-free subsets of $A = \bigcup_{1 \leq i \leq k} A_i$ contain exactly one element from each set $A_1,\ldots,A_k$?
\end{framed}

It is easy to give an fpt-Turing reduction from \paramcount{Multicolour $\LL$-Free Subset} to \countprob\, parameterised by $k$.

\begin{lemma}
\label{lma:reduce-multicol-L}
Let $\LL$ be a linear equation.  Then \paramcount{Multicolour $\LL$-Free Subset} \leqfptT \countprob (where \countprob\, is parameterised by $k$).
\end{lemma}
\begin{proof}
Let $(A_1,\ldots,A_k)$ be the input to an instance of \paramcount{Multicolour $\LL$-Free Subset}.  For each non-empty $I \subseteq [k]$, we can use our oracle to $\LL$-Free Subset to find $N_I$, the number of  $\LL$-free subsets of cardinality exactly $k$ in the set $\bigcup_{i \in I} A_i$.  This requires $\Theta(2^k)$ oracle calls, and for each oracle call the parameter value is the same as for the original problem.  We can now use an inclusion-exclusion method to compute the number of  $\LL$-free subsets of size $k$ in $A$ that contain exactly one number from each of the sets $A_i$: this is precisely
$$\sum_{\emptyset \neq I \subseteq [k]} (-1)^{k-|I|}N_I.$$
\end{proof}

The main work in the reduction is in the next lemma, where we show that \paramcount{Multicolour Clique} can be reduced to \paramcount{Multicolour $\LL$-Free Subset} for certain equations $\LL$.
\begin{lemma}
\label{lma:reduce-clique-L}
Let $\LL$ be a linear equation of the form $a_1x_1 + a_2x_2 = by$, where $a_1,a_2,b \in \mathbb{N}$ are fixed.  Then \paramcount{Multicolour Clique} \leqfptT \paramcount{Multicolour $\LL$-Free Subset}.
\end{lemma}
\begin{proof}
Let $(G,\{V_1,\ldots,V_k\})$ be the input to an instance of \paramcount{Multicolour Clique}, where $G = (V,E)$, and for each $1 \leq i < j \leq k$ let $E_{i,j}$ denote the set of edges between $V_i$ and $V_j$.  We may assume that $|E_{i,j}|=q$ for each $1 \leq i < j \leq k$ and that each $V_i$ is an independent set.

Suppose that $V = \{v_1,\ldots,v_n\}$.  We begin by constructing a set $A\subseteq \mathbb N$ as in Lemma \ref{useful1}; note that $|A| = \mathcal{O}(|V|^2)$ and $\log (\max(A)) = \mathcal{O}(|V|)$, so $|A| \log(\max (A)) = \mathcal{O}(|V|^3)$.  We partition $A'$ into $k$ subsets $A_1,\ldots,A_k$, where $A_i := \phi_V^{-1}(V_i)$, and $A''$ into $\binom{k}{2}$ subsets $A_{i,j}$ (for $1 \leq i < j \leq k$) where $A_{i,j} := \phi_E^{-1}(E_{i,j})$.

Now let $t \in \{1,\ldots,\binom{k}{2}\}$.  We define $X_t \subseteq \mathbb{N}$ to be a set of $t \binom{k}{2} < k^4$ natural numbers disjoint from $A$, chosen so that every solution to $\LL$ in $A \cup X_t$ is contained in $A$.  Without loss of generality, we may assume that $k^4 < |V|$ (otherwise we would be able to execute a brute force approach in time bounded by a function of $k$ alone), so it follows from Lemmas~\ref{extend1} and~\ref{extend2} that we can construct such a set $X_t$ in time bounded by a polynomial function of $|A| \log(\max(A))$, and hence by a polynomial function of $|V|$.  Moreover, the space required to represent $X_t$ is also bounded by a polynomial function of $|V|$.  We now partition $X_t$ arbitrarily into $\binom{k}{2}$ sets $X_t^{i,j}$ for $1 \leq i < j \leq k$, each of size exactly $t$.  The set $A_{i,j}[t]$ is then defined to be $A_{i,j} \cup X_t^{i,j}$.  We set 
$$A[t] := \bigcup_{1 \leq i \leq k} A_i \; \cup \bigcup_{1 \leq i < j \leq k} A_{i,j}[t].$$

It follows from Lemma \ref{useful1} and the construction of $A[t]$ that the only solutions to $\LL$ in $A[t]$ are of the form $a_1x_1 + a_2x_2 = by$, where $y$ corresponds to an edge whose endpoints are the vertices corresponding to $x_1$ and $x_2$.  We will say that a subset of $A[t]$ is \emph{colourful} if it contains precisely one element from each set $A_i$ (for $1 \leq i \leq k$) and one element from each set $A_{i,j}[t]$ (for $1 \leq i < j \leq k$).


Let $N(A[t])$ denote the number of $\LL$-free subsets of $A[t]$ that are colourful.  We can compute the value of $N(A[t])$ using a single call to our oracle for \paramcount{Multicolour $\LL$-Free Subset} with input $(A_1,\ldots,A_k,A_{1,2}[t],\ldots,A_{k-1,k}[t])$; note that the total size of the instance in such an oracle call is bounded by $h(k) \cdot |V|^{\mathcal{O}(1)}$ for some function $h$, and that the value of the parameter in our oracle call depends only on $k$.

Given any subset $U \subseteq V$ such that $|U \cap V_i| = 1$ for each $1 \leq i \leq k$, let us denote by $N(A[t],U)$ the number of colourful $\LL$-free subsets of $A[t]$ whose intersection with $A'$ is precisely $\phi_V^{-1}(U)$.  We now claim that 
$$N(A[t],U) = (q+t)^{\binom{k}{2}-e(U)}(q+t-1)^{e(U)},$$
where $e(U)$ denotes the number of edges in the subgraph of $G$ induced by $U$.  To see that this is true, suppose that $U = \{w_1,\ldots,w_k\}$, where $w_i \in V_i$ for each $i$.  If $w_i$ and $w_j$ are not adjacent, then we can choose freely any element of $A_{i,j}[t]$ to add to the set, without risk of creating a solution to $\LL$, so there are $\left|A_{i,j}[t]\right| = q+t$ possibilities for the element of $A_{i,j}[t]$ in the set; on the other hand, if $w_i$ and $w_j$ are adjacent, we must avoid the element of $A_{i,j}[t]$ corresponding to $w_iw_j$, so there are $\left|A_{i,j}[t]\right| - 1 = q+t -1$ possibilities for the element of $A_{i,j}[t]$ in the set.  Since we can choose each element of $A''$ to include in the set independently of the others, this gives the expression above for $N(A[t],U)$.

Observe now that
\begin{align*}
N(A[t]) \quad & = \sum_{\substack{U \subseteq V \\ \text{$|U \cap V_i| = 1$ for each $1 \leq i \leq k$}}} N(A[t],U) \\
		  & = \sum_{\substack{U \subseteq V \\ \text{$|U \cap V_i| = 1$ for each $1 \leq i \leq k$}}} (q+t)^{\binom{k}{2}-e(U)}(q+t-1)^{e(U)}.
\end{align*}
For $0 \leq j \leq \binom{k}{2}$, let $C_j$ denote the number of subsets $U \subset V$ such that $|U \cap V_i| = 1$ for each $1 \leq i \leq k$ and $e(U) = j$.  We can then rewrite the expression above as
$$N(A[t]) = \sum_{j=0}^{\binom{k}{2}} C_j (q+t)^{\binom{k}{2}-j}(q+t-1)^j.$$
If we now define
$$p(z) = \sum_{j=0}^{\binom{k}{2}} C_j z^{\binom{k}{2}-j}(z-1)^j,$$
it is clear that $p$ is a polynomial in $z$ of degree $\binom{k}{2}$, and moreover that $p(z) = N(A[z-q])$.  Thus if we know the value of $p(z)$ for $\binom{k}{2} + 1$ distinct values of $z$, we can interpolate in polynomial time to determine all the coefficients of $p$; we can obtain the required values of $p(z)$ by using our oracle to evaluate $N(A[t])$ for $t \in \{0,1,\ldots,\binom{k}{2}\}$.

To complete the reduction we must demonstrate that, once we know the coefficients of $p(z)$, it is straightforward to calculate the number of multicolour cliques in $G$.  We will in fact argue that we only need to determine the constant term of $p(z)$.  Note that, if $0 \leq j < \binom{k}{2}$, then $z$ is a factor of $C_j z^{\binom{k}{2}-j}(z-1)^j$.  Thus the constant term of $p(z)$ is the same as the constant term of the polynomial
$$C_{\binom{k}{2}}z^{\binom{k}{2}-\binom{k}{2}}(z-1)^{\binom{k}{2}} = C_{\binom{k}{2}} (z-1)^{\binom{k}{2}}.$$
This constant term is
$$C_{\binom{k}{2}} (-1)^{\binom{k}{2}},$$
so the absolute value of the constant term in $p(z)$ is precisely $C_{\binom{k}{2}}$.  But $C_{\binom{k}{2}}$ is by definition the number of subsets $U \subseteq V$ such that $|U \cap V_i| = 1$ for each $i$ and $e(U) = \binom{k}{2}$, that is the number of multicolour cliques in $G$.  

This completes the fpt-Turing reduction from \paramcount{Multicolour Clique} to \paramcount{Multicolour $\LL$-Free Subset}.
\end{proof}


 
The main result of this section now follows immediately from Lemmas~\ref{lma:reduce-multicol-L} and~\ref{lma:reduce-clique-L}.

\begin{thm}
Let $\LL$ be a linear equation of the form $a_1x_1+a_2x_2 = b y$ where $a_1,a_2,b \in \mathbb{N}$ are fixed.  Then \countprob, parameterised by $k$, is \#W[1]-hard with respect to fpt-Turing reductions.
\end{thm}

\subsection{Approximate counting}
\label{sec:counting-approx}
For $3$-variable homogeneous linear equations $\LL$,
we have seen that there is unlikely to be an fpt-algorithm, with parameter $k$, to solve \countprob\, \emph{exactly}; in this section we show, however, that the problem does admit an FPTRAS (for any non-translation-invariant homogeneous linear equation $\LL$).  The algorithm uses a simple random sampling method; the only requirement is to demonstrate that there are sufficiently many $\LL$-free subsets of size exactly $k$ that we can obtain a good estimate of the overall proportion of such subsets that are $\LL$-free without having to perform too many iterations of the sampling process.  Note that the proof of this result follows that of \cite[Lemma 3.4]{treewidth} very closely.  We write $A^{(k)}$ to denote the set of all subsets of $A$ of size $k$.

\begin{lemma}
\label{lma:dense-approx}
Let $\LL$ be any  non-translation-invariant homogeneous linear equation.   Let $A \subset \mathbb{Z}$ be finite and $k \in \mathbb{N}$, and let $N$ denote the number of elements of $A^{(k)}$ that are $\LL$-free.  Then, for every $\eps > 0$ and $\delta \in (0,1)$ there is an explicit randomised algorithm which outputs an integer $\alpha$, such that
$$\mathbb{P}[|\alpha - N| \leq \eps \cdot N] \geq 1 - \delta,$$
and runs in time at most $f(k)q(\size{A},\eps^{-1},\log(\delta^{-1}))$, where $f$ is a computable function, $q$ is a polynomial.
\end{lemma}
\begin{proof}
Let $\lambda=\lambda (\LL)>0$ be as defined in (\ref{quote}).  We begin by setting $A' := A \setminus \{0\}$: note that, as no $\LL$-free subset of $A$ can contain $0$, both $A$ and $A'$ contain precisely the same number of $\LL$-free subsets of size $k$.  We may assume throughout that $|A| \geq \frac{k}{\lambda} + 2$ and hence $|A'| \geq \frac{k}{\lambda} + 1$, otherwise we could count $\LL$-free subsets deterministically by brute force within the required time bound.

Let $N$ denote the total number of elements of $A'^{(k)}$ that are $\LL$-free: our goal is then to compute an approximation to $N$.  We do this using a simple random sampling algorithm.  At each step, a subset of $A'$ of size $k$ is chosen uniformly at random from all elements of $A'^{(k)}$; 
we then check in time $\mathcal{O}(k^{\ell}\log (\max^* (A')))$
 (where $\ell$ is the number of variables in $\LL$) whether the chosen subset is $\LL$-free.

To obtain a good estimate for $N$, we repeat this sampling process $t$ times (for some value of $t$ to be determined); we will denote by $t_1$ the number of sets selected in this way that are indeed $\LL$-free.  We then output as our approximation $\frac{t_1}{t} \binom{|A'|}{k}$.  Note that in the case that $N=0$, we are certain to output 0, as required.

The value of $t$ must be chosen to be large enough so that 
$$\mathbb{P}\left[\left| \frac{t_1}{t} \binom{|A'|}{k} - N \right|  > \eps N \right] \leq \delta,$$
or equivalently
$$\mathbb{P}\left[\left| t_1 - \frac{t N}{\binom{|A'|}{k}} \right|  > \eps \frac{t N}{\binom{|A'|}{k}} \right] \leq \delta.$$
Note that $t_1$ has distribution $\mathrm{Bin}(t,p)$, where $p := \frac{N}{\binom{|A'|}{k}}$, so the expected value of $t_1$ is exactly $\frac{t N}{\binom{|A'|}{k}}$.

Using a Chernoff bound, we therefore see that
\begin{align*}
\mathbb{P} \left[\left| t_1 - \frac{t N}{\binom{|A'|}{k}} \right|  > \eps \frac{t N}{\binom{|A'|}{k}} \right] & \leq 2 \exp \left(\frac{- \eps^2 t N}{(2+\eps)\binom{|A'|}{k}}\right), 
\end{align*}
so it is enough to ensure that
$$2 \exp \left(\frac{- \eps^2 t N}{(2+\eps)\binom{|A'|}{k}}\right) \leq \delta.$$
In other words,
\begin{align*}
\frac{- \eps^2 t N}{(2+\eps)\binom{|A'|}{k}} & \leq \log{\delta} - \log{2} \\
\iff t & \geq \frac{\binom{|A'|}{k} (2 + \eps) \left(\log{2} - \log{\delta}\right)}{\eps^2 N} \\
& = \frac{\binom{|A'|}{k}}{N} \left(2 \eps^{-2} + \eps^{-1} \right)\left(\log{2} + \log{\delta^{-1}}\right).
\end{align*}
So, in order to show that we can choose a value of $t$ that is not too large, it suffices to demonstrate that $\displaystyle\frac{\binom{|A'|}{k}}{N}$ is bounded by $f(k)\tilde{q}(\size{A})$, where $f$ is some computable function and $\tilde{q}$ is a polynomial.

However, by Theorem~\ref{lamthm2} we know that $A'$ contains an $\LL$-free subset $B$ of size at least $\lambda |A'|$, and any subset of $B$ is necessarily $\LL$-free; therefore $N \geq \binom{\lambda |A'|}{k}$.  Hence
$$\frac{\binom{|A'|}{k}}{N} \leq \frac{\binom{|A'|}{k}}{\binom{\lambda |A'|}{k}} \leq \left(\frac{|A'|}{\lambda|A'|-k}\right)^k = \left( \frac{\frac{1}{\lambda}|A'|}{|A'|- \frac{1}{\lambda}k}\right)^k = \left( \frac{1}{\lambda} + \frac{\frac{1}{\lambda^2}k}{|A'|-\frac{1}{\lambda}k}\right)^k \leq \left(\frac{1}{\lambda} + \frac{1}{\lambda^2}k \right)^k,$$
where the last inequality follows from our initial assumption that $|A'| \geq \frac{k}{\lambda} + 1$.  Hence $\displaystyle\frac{\binom{|A'|}{k}}{N}$ is bounded by a function of $k$, as required.
\end{proof}

The following result now follows immediately from Lemma \ref{lma:dense-approx}.

\begin{thm}
Let $\LL$ be any  non-translation-invariant homogeneous linear equation.   Then \countprob admits an FPTRAS.
\end{thm}

\section{An extension version of the problem}\label{sec:ex}

A natural variant of the problem $\LL$-\textsc{Free Subset} is to ask whether, given $A \subseteq \mathbb{Z}$ and an ($\LL$-free) subset $B \subset A$, there is an $\LL$-free subset of $A$ of cardinality $k$ which contains $B$.  This problem can be stated formally as follows.

\begin{framed}
\noindent $\LL$-\textsc{Free Subset Extension}\newline
\textit{Input:} A finite set $A \subseteq \mathbb Z$, a set $B \subset A$ and $k \in \mathbb N$.\newline
\textit{Question:} Does there exist an $\LL$-free subset $A' \subseteq A$ such that $B \subseteq A'$ and $|A'|=k$?
\end{framed}

We can make certain easy deductions about the complexity of this problem from the results we have already proved about $\LL$-\textsc{Free Subset}.  Notice that we can easily define a reduction from $\LL$-\textsc{Free Subset} to $\LL$-\textsc{Free Subset Extension} by setting $B = \emptyset$; the next result follows immediately from this observation together with Theorem \ref{npthm}.

\begin{prop}
Let $\LL$ be a linear equation of the form $a_1x_1+\dots+a_\ell x_\ell = b y$ where each $a_i \in \mathbb N$ and $b \in \mathbb N$ are fixed and $\ell\geq 2$.
Then $\LL$-\textsc{Free Subset Extension} is $\NP$-complete, and the problem is para-$\NP$-complete parameterised by $|B|$.
\end{prop}

Perhaps the most obvious parameterisation of this problem to consider is $k - |B|$, the number of elements we want to add to the set $B$.  It is straightforward to adapt our earlier results to demonstrate that, in the case of three-term equations, the problem is unlikely to  admit an fpt-algorithm with respect to this parameterisation.

\begin{prop}
Let $\LL$ be a linear equation of the form $a_1x_1 + a_2x_2 = by$, where $a_1,a_2,b \in \mathbb{N}$ are fixed.  Then $\LL$-\textsc{Free Subset Extension} is W[1]-hard, parameterised by $k - |B|$.
\end{prop}
\begin{proof}
We prove this result by means of a reduction from the W[1]-complete problem \paramdec{Independent Set}, which is defined as follows.
\begin{framed}
\noindent \paramdec{Independent Set} \\
\textit{Input:} A graph $G$ and $k \in \mathbb{N}$.\\
\textit{Parameter:} $k$\\
\textit{Question:} Does $G$ contain an independent set of cardinality $k$?
\end{framed}
The W[1]-completeness of this problem can easily be deduced from that of \paramdec{Clique}, shown to be W[1]-complete in \cite{downey95}.

Let $(G,k')$ be the input to an instance of \paramdec{Independent Set}.  Once again, we rely on the construction in Lemma \ref{useful1} to give us the set $A$ in our instance of $\LL$-\textsc{Free Subset Extension}; we set $B := A''$ and $k := |B| + k'$ (so the parameter of interest in our instance of $\LL$-\textsc{Free Subset Extension} is equal to the solution size in the instance of \paramdec{Independent Set}).  By Corollary \ref{cor:is-reduction}, we know that there is a one-to-one correspondence between independent sets in $G$ of cardinality $k'$ and $\LL$-free subsets of $A$ of cardinality $|A''| + k'$ that contain $A''$, so it follows immediately that $(A,B,k)$ is a yes-instance for $\LL$-\textsc{Free Subset Extension} if and only if $(G,k')$ is a yes-instance for \paramdec{Independent Set}.  
\end{proof}

On the positive side, we observe that we can once again make use of fpt-algorithms for \textbf{p}-card-\textsc{Hitting Set} if we consider $\LL$-\textsc{Free Subset Extension} parameterised by the number of elements of $A$ that are \emph{not} included in the subset.  Note that the standard bounded search tree method for \textbf{p}-card-\textsc{Hitting Set} (or its counting version) \cite[Theorem 1.14]{flumgrohe} in fact gives an fpt-algorithm to find all hitting sets of size $k$.  As it is easy to check in time polynomial in $\size{A}$ whether a given hitting set in the hypergraph defined in the proof of Lemma \ref{lma:to-HS} contains a vertex corresponding to an element of $B$, we can use this method to count $\LL$-free subsets of $A$ that contain $B$ (and hence to decide whether there is at least one).

\begin{prop}
Let $\LL$ be any fixed linear equation. Then $\LL$-\textsc{Free Subset Extension}, parameterised by $|A|-k$, belongs to $\FPT$; the same is true for the counting version of the problem with this parameterisation.
\end{prop}

Finally, we consider parameterising simultaneously by the number of elements we wish to add to $B$ and the size of the set $B$; this is equivalent to parameterising by $k$, the total size of the desired $\LL$-free subset.

\begin{prop}
Let $\LL$ be a linear equation of the form $a_1x_1 + a_2x_2 = by$, where $a_1,a_2,b \in \mathbb{N}$ are fixed and $a_1 + a_2 \neq b$.  Then $\LL$-\textsc{Free Subset Extension}, parameterised by $k$, belongs to $\FPT$.
\end{prop}
\begin{proof}
Let $(A,B,k)$ be an instance of $\LL$-\textsc{Free Subset Extension}.  If $k < |B|$ then this is necessarily a no-instance, so we may assume without loss of generality that $k \geq |B|$; we may also assume that $B$ is $\LL$-free (we can check this in polynomial time and if $B$ contains a solution to $\LL$ we immediately return NO).

As a first step in our algorithm, we delete from $A$ every $a \in A \setminus B$ such that $B \cup \{a\}$ is not $\LL$-free: note that this does not change the number of $\LL$-free subsets containing $B$, as such an $a$ cannot belong to any set of this kind.  We call the resulting set $A_1$, and note that we can construct $A_1$ in time polynomial in $\size{A}$. Note that no set containing $0$ can be $\LL$-free, so we know that $0 \not \in A_1$. 

Fix the constant $\lambda = \lambda(\LL)$ as in equation \eqref{quote}.  Our algorithm proceeds as follows.  
If $|A_1| < \displaystyle \left(\frac{6|B| + 1}{\lambda}\right)(k - |B|) + |B|$, we exhaustively consider all $k$-element subsets of $A_1$ and check if they form an $\LL$-free subset;
if $|A_1| \geq \displaystyle \left(\frac{6|B| + 1}{\lambda}\right)(k - |B|) + |B|$ then we return YES.  To see that this is an fpt-algorithm, note that, if $|A_1| < \displaystyle \left(\frac{6|B| + 1}{\lambda}\right)(k - |B|) + |B|$, then $|A_1| = \mathcal{O}(k^2)$, so we can perform the exhaustive search in time depending only on $k$.  It is clear that we will return the correct answer whenever we perform the exhaustive search; in order to prove correctness of the algorithm, it remains to show that, if $|A_1| \geq \displaystyle \left(\frac{6|B| + 1}{\lambda}\right)(k - |B|) + |B|$, then we must have a yes-instance.

To see that this is true, we first prove that there exists a large set $A_2 \subseteq A_1$ such that $B \subseteq A_2$ and no solution to $\LL$ in $A_2$ involves an element of $B$.  We will call a triple $(x,y,z )\in A_1^3$ \emph{bad} if $a_1x + a_2y = bz$.  By construction of $A_1$, note that every bad triple contains at most one element of $B$.  We aim to bound the number of bad triples containing some fixed $u \in A_1 \setminus B$ and at least one element of $B$.  First, we also fix $v \in B$, and bound the number of bad triples which contain both $u$ and $v$.  If we fix the positions of $u$ and $v$ in a triple, there is at most one $w \in A_1 \setminus B$ such that $w$ completes the triple; as there are 6 options for the choice of positions of $u$ and $v$ in the triple, this means there are in total at most $6$ bad triples involving both $u$ and $v$.  Summing over all possibilities for $v$, we see that there are at most $6|B|$ bad triples involving any fixed $u \in A_1 \setminus B$ and at least one element of $B$.

We can therefore greedily construct a set $C\subseteq A_1 \setminus B$ of size at least $\frac{|A_1| - |B|}{6|B|+1}$ such that every bad triple in $B \cup C$ is entirely contained in $C$.  
Indeed, initially set $C:=\emptyset$ and $A':=A_1\setminus B$. Move an arbitrary element $u$ of $A'$ to $C$ and delete all elements of $A'$ that form a bad triple with $u$ and at least one element of $B$; by the reasoning above, this involves deleting at most $6|B|$ elements of $A'$. Repeat this process until $A'=\emptyset$, and note that $C$ is as desired; set $A_2 := B \cup C$.

Now observe that every $\LL$-free subset of cardinality $k - |B|$ in $C$ can be extended to an $\LL$-free subset of $A_2$ of cardinality $k$ which contains $B$.  We know from Theorem \ref{lamthm2} that there exists an $\LL$-free subset $C' \subseteq C$ of cardinality at least $\lambda |C|$ (where $\lambda$ is the constant defined in equation (\ref{quote})).  $B \cup C'$ is then an $\LL$-free subset of $A$, and has cardinality at least 
$$|B| + \lambda \left(\frac{|A_1|-|B|}{6|B|+1}\right) \geq |B| + \lambda\left(\frac{\left(\frac{6|B| + 1}{\lambda}\right)(k - |B|) + |B| - |B|}{6|B| + 1}\right) = k,$$
so we should indeed return YES.
\end{proof}

We note that the argument used in this proof can be adapted to demonstrate the existence of an FPTRAS for the counting version of this problem, using the ideas from Lemma \ref{lma:dense-approx}.  However, there is unlikely to be an fpt-algorithm to solve the counting version exactly, as we can easily reduce \countprob to this problem (with the same parameter) by setting $B = \emptyset$.

\section{Conclusions and open problems}

We have shown that the basic problem of deciding whether a given input set $A \subseteq \mathbb{Z}$ contains an $\LL$-free subset of size at least $k$ is $\NP$-complete when $\LL$ is any linear equation of the form $a_1x_1 + \cdots + a_{\ell}x_{\ell} = by$ (with $a_i,b \in \mathbb{N}$ and $\ell \geq 2$), although the problem is solvable in polynomial time whenever $\LL$ is a linear equation with only two variables.  We also demonstrated that the maximisation version of the problem is $\APX$-hard for equations $\LL$ of the form $a_1x_1 + a_2x_2 = by$ (with $a_1,a_2,b \in \mathbb{N}$).  

Two natural questions arise from these results.  First of all, in our $\NP$-hardness reduction, we construct a set $A$ where $\max(A)$ is exponential in terms of $|A|$: is this problem in fact \emph{strongly} $\NP$-complete, so that it remains hard even if all elements of $A$ are bounded by some polynomial function of $|A|$?  Secondly, can either the $\NP$-completeness proof or the $\APX$-hardness proof be generalised to other linear equations $\LL$?  A natural starting point for an equation that is not covered by Theorem \ref{npthm} would perhaps be the case of Sidon sets (i.e. $x+y=z+w$).

On the positive side, we saw that the decision problem belongs to $\FPT$ for any homogeneous non-translation-invariant equation $\LL$ when parameterised by the cardinality of the desired $\LL$-free subset, and that it belongs to $\FPT$ for any linear equation $\LL$ with respect to the dual parameterisation (the number of elements of $A$ not included in the $\LL$-free subset).  While we have considered two natural parameterisations here, there is another natural parameterisation that we have not considered.  We know that, for certain linear equations, there is  some function $c^*_{\LL}$ such that every set $A \subseteq \mathbb{Z}$ is certain to contain an $\LL$-free subset of cardinality at least $c^*_{\LL}(|A|)$.  It is therefore natural to consider the complexity parameterised above this lower-bound: what is the complexity of determining whether a given subset $A \subseteq \mathbb{Z}$ contains an $\LL$-free subset of cardinality at least $c^*_{\LL}(|A|) + k$, where $k$ is taken to be the parameter?  The main difficulty in addressing this question is that the exact value of $c^*_{\LL}$ is not known for any linear equations $\LL$: even in the case of sum-free subsets, we only know that the bound on $c^*_{\LL}$ is of the form $\frac{|A|}{3} + o(|A|)$.

We also considered the complexity of determining whether a set $A \subseteq \mathbb{Z} \setminus \{0\}$ contains an $\LL$-free subset containing a fixed proportion $\eps$ of its elements.  We demonstrated that this problem is also $\NP$-complete for the case of sum-free sets, and also for $\LL$-free sets whenever $\LL$ is a $3$-variable translation-invariant linear equation.  It would be interesting to investigate how far these results can be generalised to other linear equations: given any non-translation-invariant, homogeneous linear equation $\LL$ and any rational $\kappa (\LL) <\eps<1$, is $\eps$-$\LL$-\textsc{Free Subset} $\NP$-complete?

Concerning the complexity of counting $\LL$-free subsets, we have addressed the problem of counting $\LL$-free subsets containing exactly $k$ elements.  For equations $\LL$ covered by the $\NP$-hardness result of Theorem \ref{npthm}, even approximate counting is hard: there is no FPRAS for arbitrary $k$ unless $\NP=\RP$ (as if we could count approximately we could, with high probability, determine whether there is at least one $\LL$-free subset of size $k$).  We also considered the complexity of this problem parameterised separately by $k$ and $|A|-k$.

However, there are other natural counting problems we have not addressed here.  For example, we might want to count the total number of $\LL$-free subsets of \emph{any} size; here the decision problem (``Is there an $\LL$-free subset of any size?'') is trivial, so there is no immediate barrier to an efficient counting algorithm.  Alternatively, we might want to count the total number of \emph{maximal} $\LL$-free sets.  Our results do not have any immediate consequences for either of these problems, but the corresponding counting versions of the extension problem are necessarily \#P-complete: by Corollary \ref{cor:is-reduction}, we have a polynomial-time reduction to this problem from that of counting all (maximal) independent sets in an arbitrary $\ell$-uniform hypergraph; for $\ell = 2$ and $\ell=3$ this problem was shown to be \#P-complete by Greenhill \cite{greenhill00}, and we can easily add further dummy vertices to each edge (and then require that the elements corresponding to the dummy vertices are included in our $\LL$-free subset) to deal with larger values of $\ell$.

{\footnotesize \obeylines \parindent=0pt
\begin{tabular}{lll}
Kitty Meeks 									&\ & Andrew Treglown\\
School of Computing Science		&\ & School of Mathematics\\
Sir Alwyn Williams Building		&\ & University of Birmingham\\
University of Glasgow					&\ & Edgbaston\\
Glasgow												&\ & Birmingham\\
G12 8QQ												&\ & B15 2TT\\
UK 														&\ & UK
\end{tabular}
}
\begin{flushleft}
{\emph{E-mail addresses}:
\tt{kitty.meeks@glasgow.ac.uk, a.c.treglown@bham.ac.uk}}
\end{flushleft}


\begin{thebibliography}{99}


\bibitem{aks}
Manindra Agrawal, Neeraj Kayal and Nitin Saxena, PRIMES is in P, \emph{Ann. of Math.} (2004) 160(2): 781-793.

\bibitem{alimonti00}
Paola Alimonti and Viggo Kann, Some APX-completeness results for cubic graphs, \emph{Theor. Comput. Sci.} (2000) 237(1-2): 123-134.

\bibitem{alonkleitman} N. Alon and D. J. Kleitman, Sum-free subsets, in \emph{A Tribute to Paul
Erd\H{o}s}, Cambridge Univ. Press, Cambridge, 1990,  13--26.


\bibitem{alonspencer} N. Alon and J.H. Spencer, The Probabilistic Method, John Wiley \& Sons, 2004

\bibitem{sos} L. Babai, V. S\'os, Sidon sets in groups and induced subgraphs of Cayley graphs, \emph{European
J. Combin.}, 6, (1985), 101--114.

\bibitem{bls} J. Balogh,  H. Liu,  M. Sharifzadeh,   The number of subsets of integers with no $k$-term arithmetic progression, \emph{Int. Math. Res. Not.}, to appear.

\bibitem{BLST}
 J.~Balogh, H.~Liu, M.~Sharifzadeh and A.~Treglown,
 \newblock The number of maximal sum-free subsets of integers,
\newblock  {\em  Proc. Amer. Math. Soc.}, 143, (2015),  4713--4721. 


\bibitem{BLST2}
 J.~Balogh, H.~Liu, M.~Sharifzadeh and A.~Treglown,
 \newblock Sharp bound on the number of maximal sum-free subsets of integers,
\newblock  submitted.


\bibitem{container1} J. Balogh, R. Morris and W. Samotij,   Independent sets in hypergraphs, \emph{J. Amer. Math. Soc.}  {\bf 28} (2015),  669--709.       



\bibitem{bloom} T.F. Bloom, A quantitative improvement for Roth's theorem on arithmetic progressions, \emph{J. London Math. Soc.}, {\bf 93} (3), (2016), 643--663.

\bibitem{bourg} J. Bourgain, Estimates related to sumfree subsets of sets of integers, \emph{Israel
J. Math.}, 97, (1997), 71--92.

\bibitem{brown} T.C. Brown and J.C. Buhler, A density version of a geometric Ramsey theorem, \emph{J. Combin.
Theory, Ser. A} {\bf 32} (1982), 20--34.

\bibitem{cam1}
P.~Cameron and P.~Erd{\H{o}}s,
\newblock On the number of sets of integers
with various properties,
\newblock in Number
Theory (R.A. Mollin, ed.), 61--79, Walter de Gruyter, Berlin, 1990.

\bibitem{CE}
P.~Cameron and P.~Erd{\H{o}}s,
\newblock Notes on sum-free and related sets,
\newblock {\em Combin. Probab. Comput.}, 8, (1999), 95--107.





\bibitem{yap} P.H. Diananda and H.P. Yap, Maximal sum-free sets of elements of finite groups, \emph{Proc. Japan Acad.}, {\bf 45}, (1969), 1--5.

\bibitem{downey95}
R.G. Downey and M.R. Fellows. Fixed-parameter tractability and completeness
II: On completeness for W[1]. Theoretical Computer Science, 141:109–131,
1995.


\bibitem{downeyfellows13}
Rodney~G. Downey and Michael~R. Fellows, \emph{Fundamentals of Parameterized Complexity}, \newblock Springer London, 2013.

\bibitem{egm} S. Eberhard, B. Green and F. Manners,
Sets of integers with no large sum-free subset, \emph{Ann. Math.}, 180, (2014), 621--652.	


\bibitem{elk} M. Elkin, An improved construction of progression-free sets, \emph{Israel J.  Math.},  184, (2011), 93--128.

\bibitem{erdos} P. Erd\H{o}s, Extremal problems in number theory, \emph{Proceedings of the Symp. Pure Math.} VIII AMS (1965),
181--189.


\bibitem{fellows09}
M.~Fellows, D.~Hermelin, F.~Rosamond, and S.~Vialette, \emph{On the
  parameterized complexity of multiple-interval graph problems}, Theoretical
  Computer Science \textbf{410} (2009), 53--61.
  
\bibitem{flum04}
J.~Flum and M.~Grohe, \emph{The parameterized complexity of counting problems},
  SIAM Journal on Computing \textbf{33} (2004), no.~4, 892--922.
  
\bibitem{flumgrohe}
J.~Flum and M.~Grohe, \emph{Parameterized complexity theory}, Springer, 2006.

\bibitem{fgr2} P. Frankl, G Graham, and V. R\"odl, On subsets of abelian groups with no $3$-term arithmetic
progression, \emph{J. Combin. Theory, Ser. A} {\bf 45} (1987), 157--161.



\bibitem{froese16}
Vincent Froese, Iyad Kanj, Andr\'e Nichterlein and Rolf Niedermeier, Finding Points in General Position, In Proceedings of the 28th Canadian Conference on Computational Geometry (CCCG '16), pages 7–14. 2016.

\bibitem{garey} M.R. Garey and D.S. Johnson, Computers and intractability, Freeman, 1979.

      
\bibitem{G-CE}
B.~Green,
\newblock The {C}ameron-{E}rd{\H o}s conjecture,
\newblock {\em Bull. London Math. Soc.}, 36, (2004), 769--778.   

 



\bibitem{GR-g}
B. Green and I. Ruzsa,
\newblock
Sum-free sets in abelian groups,
\newblock {\em Israel J. Math.},  147, (2005), 157--189.   

\bibitem{greenwolf} B.  Green and J. Wolf, A note on Elkin’s improvement of Behrend's construction,
\emph{Additive number theory: Festschrift in honor of the sixtieth birthday of Melvyn B. Nathanson}, pages
141--144. Springer-Verlag, 1st edition, 2010.

\bibitem{greenhill00} C. Greenhill, The complexity of counting colourings and independent sets in sparse graphs and hypergraphs, \emph{Comput. complex.} (2000) 9:52, doi:10.1007/PL00001601.


\bibitem{gow} W.T. Gowers, Quasirandom groups, \emph{Combin. Probab. Comput.}, 17, (2008), 363--387.



\bibitem{kolo} M.N. Kolountzakis,
Selection of a large sum-free subset in polynomial time, \emph{Inform. Process. Lett.} {\bf 49} (1994),
255--256.






\bibitem{lev} V.F. Lev, Progression-free sets in finite abelian groups, \emph{J. Number Theory} {\bf 104} (2004), 
162--169.

\bibitem{treewidth}
Kitty Meeks, The challenges of unbounded treewidth in parameterised
  subgraph counting problems, \emph{Discrete Applied Mathematics} \textbf{198} (2016), pp.170-194; doi:10.1016/j.dam.2015.06.019.



\bibitem{niedermeier03}
Rolf Niedermeier, Peter Rossmanith, An efficient fixed-parameter algorithm for 3-Hitting Set, \emph{Journal of Discrete Algorithms}, Volume 1, Issue 1, February 2003, Pages 89-102, ISSN 1570-8667, http://dx.doi.org/10.1016/S1570-8667(03)00009-1.



\bibitem{roth} K.F.
Roth, On certain sets of integers,
{\em J. London Math. Soc.},  28, (1953), 104--109.

\bibitem{ruzsa} I.Z. Ruzsa, Solving a linear equation in a set of integers I, \emph{Acta Arith.}, 65, (1993), 259--282.


\bibitem{ruzsa2} I.Z. Ruzsa, Solving a linear equation in a set of integers II, \emph{Acta Arith.}, 72, (1995), 385--397.

\bibitem{sand} T. Sanders, On Roth's theorem on progressions, \emph{ Ann. of Math.}, 174, (2011),  619--636.


\bibitem{sap} A.A.~Sapozhenko,
\newblock The {C}ameron-{E}rd{\H o}s conjecture, (Russian)
\newblock {\em Dokl. Akad. Nauk.}, 
393,
(2003), 
749--752.

\bibitem{container2} D. Saxton and A. Thomason, Hypergraph containers, \emph{Invent. Math.} {\bf 201} (2015), 925--992.











\bibitem{schur}  I. Schur, 
\newblock Uber die Kongruenz $x^m +y^m \equiv z^m$ (mod $p$),
\newblock {\em ber. Deutsch. Mat. Verein.}, 25,
(1916), 114--117.

\bibitem{thurley07}
M. Thurley, Kernelizations for Parameterized Counting Problems, In: Cai JY., Cooper S.B., Zhu H. (eds) Theory and Applications of Models of Computation. TAMC 2007. Lecture Notes in Computer Science, vol 4484. Springer, Berlin, Heidelberg.




\bibitem{sz} E. Szemer\'edi, On sets of integers containing no $k$ elements in arithmetic progression, \emph{Acta Arith.},
{\bf  27} (1975), 199--245.


\bibitem{taovu} T. Tao and V. Vu, Sumfree sets in groups: a survey, submitted.
\bibitem{vdw} B.L. van der Wareden, Beweis einer Baudetschen Vermutung, \emph{Nieuw Arch. Wisk.} {\bf 15} (1927),
212--216.

\end{thebibliography}
\end{document}